\documentclass[10pt]{article}
\usepackage[utf8]{inputenc}
\usepackage{amsmath, amsthm, amssymb, mathrsfs, amscd, amsthm, amscd,amsfonts,calligra,mathrsfs,adjustbox,lipsum}
\usepackage[textwidth=360pt,textheight=615pt]{geometry}
\usepackage{indentfirst,url}
\usepackage[all]{xy}
\usepackage{url}
\usepackage{tikz}
\usepackage{paracol}

\newtheorem{Theorem}{Theorem}
\newtheorem{Proposition}{Proposition}
\newtheorem{corollary}{Corollary}

\newtheorem{Lemma}{Lemma}
\newtheorem{Notation}[Proposition]{Notation}

\DeclareMathOperator{\glct}{glct}
\DeclareMathOperator{\lct}{lct}
\DeclareMathOperator{\mult}{mult}
\DeclareMathOperator{\support}{supp}
\newtheorem*{Acknowledgments}{ACKNOWLEDGMENTS}

\usepackage{CJKutf8}

\usepackage{amsmath,amssymb,amscd,amsfonts}

\title{Log canonical thresholds of Burniat surfaces with $K^2 = 5$
}
\author{Nguyen Bin, Jheng-Jie Chen and YongJoo Shin}
\date{}

\newcommand{\BinAddresses}{{
		\bigskip
		\footnotesize
        \text{Nguyen Bin,}\par\nopagebreak	
		\text{Department of Mathematics and Statistics,}\par\nopagebreak	
		\text{Quy Nhon University,}\par\nopagebreak	
            \text{170 An Duong Vuong Street, Quy Nhon,}\par\nopagebreak
		\text{Vietnam.}\par\nopagebreak		
		\textit{E-mail address}: \texttt{nguyenbin@qnu.edu.vn}
		
}}

\newcommand{\JhengJieAddresses}{{
		\bigskip
		\footnotesize
        \text{Jheng-Jie Chen,}\par\nopagebreak	
		\text{Department of Mathematics,}\par\nopagebreak	
		\text{National Central University,}\par\nopagebreak	
            \text{No. 300, Zhongda Rd., Zhongli District, }\par\nopagebreak
		\text{Taoyuan City, Taiwan.}\par\nopagebreak		
		\textit{E-mail address}: \texttt{jhengjie@math.ncu.edu.tw}
		
}}

\newcommand{\YongJooAddresses}{{
		\bigskip
		\footnotesize
        \text{YongJoo Shin,}\par\nopagebreak	
		\text{Department of Mathematics,}\par\nopagebreak	
		\text{Chungnam National University,}\par\nopagebreak	
            \text{99, Daehak-ro, Yuseong-gu, }\par\nopagebreak
		\text{Daejeon, Republic of Korea.}\par\nopagebreak		
		\textit{E-mail address}: 
        \texttt{haushin@cnu.ac.kr}
		
}}

\begin{document}
\maketitle
\begin{abstract}
     In the paper we compute the global log canonical thresholds of the secondary Burniat surfaces with $K^2 = 5$. Furthermore, we establish optimal lower bounds for the log canonical thresholds of members in pluricanonical sublinear systems of the secondary Burniat surfaces with $K^2 = 5$.

\end{abstract}
\section{Introduction}

    This work was intended as an attempt to compute global log canonical thresholds of complex algebraic projective surfaces of general type. The study of surfaces of general type has a rich history dating back to the 19$^{\text{th}}$ century, and understanding their singularities is a classical topic in the theory of algebraic surfaces. In the last decades, the problem of computing global log canonical thresholds of complex algebraic projective surfaces has attracted the attention of many authors. Despite a significant progress in understanding the global log canonical thresholds of del Pezzo surfaces and higher dimensional Fano varieties, there is a scarcity of works addressing this topic for surfaces of general type.\\
	
	Let $X$ be a normal algebraic surface with at most log canonical singularities, and let $D$ be an effective $\mathbb{Q}$-Cartier divisor on $X$. The discrepancy of the log pair $(X,D)$ serves as a measure of its singularity. It is important to note that if $(X,D)$ is not log canonical, the discrepancy of $(X,D)$ is $-\infty$. In such case, the discrepancy provides no additional information beyond the non-log canonicity of $(X,D)$. Therefore the concept of the log canonical threshold is introduced to distinguish among non-log canonical singularities. The log canonical threshold of $X$ with respect to $D$, denoted as $\lct(X, D)$, is defined as follows:	
	$$ \lct\left( X, D\right) :=\sup\left\lbrace \lambda \in \mathbb{Q} \mid \text{the log pair} \left( X,\lambda D\right) \text{is log canonical}  \right\rbrace.  $$
	\noindent
	On the other hand, the global log canonical threshold of a polarized variety $X$ by a $\mathbb{Q}$-Cartier divisor $L$, denoted by $\glct(X, L)$, is defined as follows:	
	\begin{equation*}
	    \glct\left( X, L\right) :=\inf\left\lbrace \lct\left( X, D\right) \begin{array}{|l} 
         D \text{ is an effective } \mathbb{Q} \text{-Cartier divisor on } X\\
        \mathbb{Q}\text{-linearly equivalent to } L 
        \end{array}
       \!\!\! \right\rbrace\!.
	\end{equation*}

	\noindent
	To compute global log canonical thresholds is a challenging task, and despite the development of various methods, it remains difficult in general. However, these thresholds play a significant role in the study of birational geometry. Notably, the global log canonical threshold of del Pezzo surfaces polarized by the anti-canonical divisor serves as an algebraic counterpart to the $\alpha$-invariant of algebraic surfaces introduced in \cite{zbMATH05528013}. It is worth noting that the global log canonical threshold of surfaces of general type via their canonical divisors finds relevance in the study of three-folds of general type \cite{MR4105534}.\\

	A surface of general type  $ X $ satisfies the Castelnuovo's inequality (1905):
	\begin{align*}
		1 -q\left( X\right) + p_g\left( X\right) \ge 1,
	\end{align*}
	\noindent
	where $ p_g\left( X\right) $ is the geometric genus and $ q\left( X\right) $ is the irregularity. Currently, computing the global log canonical threshold of all surfaces of general type is far from reach. Hence, it is natural to focus on computing the global log canonical threshold of surfaces of general type that satisfy Castelnuovo's line. In particular, we will direct our attention to surfaces with $p_g = q = 0$. Among these surfaces, Burniat surfaces have been the subject of intense research. Burniat surfaces were constructed by Pol Burniat in 1966 \cite{zbMATH03232556}, and they have the self-intersection number $K^2 = 6, 5, 4, 3, 2$ of the canonical divisor $K$. The Burniat surfaces with $K^2 = 6$ are called the primary Burniat surfaces \cite{MR2609250}. The global log canonical threshold of the primary Burniat surface polarized by the canonical divisor was computed by Kim and  the third author \cite{MR4216579}, who proved that it is $\frac{1}{2}$. Our objective is to study the log canonical thresholds of the secondary Burniat surfaces with $K^2 = 5$ \cite{MR2609250}. A main result of this paper is as follows:	
	\begin{Theorem}\label{MainTheorem1}
		Let $ X $ be a secondary Burniat surface with $ K^2_X = 5 $. Then 
		\begin{align*}
			\glct\left( X,K_X\right) = \frac{1}{2}.
		\end{align*}
	\end{Theorem}

	\noindent
	Our previous theorem demonstrates that the secondary Burniat surfaces with $K^2 = 5$ possess the same global log canonical threshold as the primary Burniat surfaces, despite the fact that the secondary Burniat surfaces  with $K^2 = 5$ are a degeneration of the primary Burniat surfaces.\\

    Since Burniat surfaces $X$ are $ \mathbb{Z}/2\mathbb{Z}\times \mathbb{Z}/2\mathbb{Z} $-covers of weak del Pezzo surfaces, the natural action of the group $ \mathbb{Z}/2\mathbb{Z}\times \mathbb{Z}/2\mathbb{Z} $ leads to the following decomposition of eigenspaces (cf. \cite[Section 2]{MR1718139}) for any positive integer $m$:
    $$H^{0}(X, mK_X) = H^{0}(X, mK_X)_0 \oplus  \bigoplus_{i=1}^3 H^{0}(X, mK_X)_i. $$
    Within the pluricanonical linear system $|mK_X|$ for any positive integer $m$, there exist an invariant part $|mK_X|_0$ and anti-invariant parts $|mK_X|_i$ for $i = 1,2,3$. The former consists of zeros of sections belonging to $H^{0}(X, mK_X)_0$, while the latter consists of zeros of sections of $H^{0}(X, mK_X)_i$ for $i = 1,2,3$.

    In their paper \cite{MR4515703}, Kim and the third author provided optimal lower bounds for the log canonical thresholds of members in pluricanonical sublinear systems for primary Burniat surfaces. Motivated by their exceptional research, we extend this line of the inquiry to secondary Burniat surfaces and present our own results. The optimal lower bounds we obtained for the log canonical thresholds of members in pluricanonical sublinear systems for secondary Burniat surfaces  with $ K^2_X = 5 $ are as follows:
	\begin{Theorem}\label{MainTheorem-2}
		Let $ X $ be a secondary Burniat surface with $ K^2_X = 5 $. Then for a positive integer $n\geq 2$ and each $i = 1,2,3$, if $D_{0} \in \left| 2nK_X\right|_0  $ and  $D_{i} \in \left| 2nK_X\right|_i  $,
		\begin{enumerate}
			\item $\lct\left( X,D_{0}\right) \ge \frac{1}{4n}$;
			\item $\lct\left( X,D_{i}\right) \ge \frac{1}{4n-3}$.
		\end{enumerate}
        Moreover, the inequalities are optimal.
	\end{Theorem}
	\noindent
     It is worth pointing out that $n=2$ is the minimal value so that each anti-invariant linear system $\left| 2nK_X\right|_i $ is non-empty (cf. Lemma \ref{pluricanonical_system_bidouble_covers}).
	\begin{Theorem}\label{MainTheorem-3}
		Let $ X $ be a secondary Burniat surface with $ K^2_X = 5 $. Then for a positive integer $n\geq 1$ and each $i = 1,2,3$, if $D_{0} \in \left| (2n+1)K_X\right|_0  $ and  $D_{i} \in \left| (2n+1)K_X\right|_i  $,
		\begin{enumerate}
			\item $\lct\left( X,D_{0}\right) \ge \frac{1}{4n-3}$;
			\item $\lct\left( X,D_{i}\right) \ge \frac{1}{4n}$.
		\end{enumerate}
        Furthermore, the inequalities are optimal.
	\end{Theorem}
	\noindent
	
     Theorems \ref{MainTheorem-2} and \ref{MainTheorem-3} highlight a clear differentiation between the lower bounds of the log canonical thresholds in the anti-invariant parts of secondary Burniat surfaces with $ K^2_X = 5 $ and primary Burniat surfaces. In particular we obtain the following.
     \begin{corollary}
         For every positive integer $n$, there is no divisor $D$ in the  linear system $|(2n+1) K_X|_i$  satisfying $\glct(X,K_X)=\lct(X, \frac{1}{2n+1}D)$ for $i=0,1,2,3$ (resp. $|2nK_X|_i$ with satisfying $\glct(X,K_X)=\lct(X, \frac{1}{2n}D)$ for $i=1,2,3$).
     \end{corollary} 

    The strategy of our proof follows the approach adopted by Kim and the third author \cite{MR4216579, MR4515703}. The inequalities presented in the above theorems are established using a method of a contradiction. It's worth highlighting, however, that the branch locus of the secondary Burniat surface with $ K^2_X = 5 $ is not symmetric. As a result, we are faced with more cases to consider. In the most cases, the inversion of adjunction formula is applied to get contradictions. Furthermore, there are specific cases where we cannot directly derive a contradiction by applying the inversion of the adjunction formula. To address this challenge, we adopt the recent work from Jiang and Zou \cite{MR4597209}. These two approaches allow us to effectively achieve the desired contradictions.
    \medskip
    
    Throughout this paper all surfaces are projective algebraic over a field of the complex numbers. Linear equivalence of divisors is denoted by $ \sim $ and $\mathbb{Q}$-linear equivalence is denoted by $\sim_{\mathbb{Q}}$. We shall denote by $\Omega$ some effective $\mathbb{Q}$-Cartier divisor on $X$ in this paper.   

\section{Preliminary}
	The standard work on the secondary Burniat surfaces is \cite{MR2609250}. For the sake of the completeness, we present the construction of the secondary Burniat surface with $ K^2 = 5 $.
	
	\begin{Notation}\label{Notation of del Pezzo surface of degree 5}
	We denote by $ Y $ the blow-up of $ \mathbb{P}^2$ at four points $ P_1, P_2, P_3, P_4 $ in general position. Let us denote by $ l $ the pull-back of a general line in $ \mathbb{P}^2$, by $ e_1 $, $ e_2 $, $ e_3 $, $ e_4 $ the exceptional divisors corresponding to $ P_1 $, $ P_2 $, $ P_3 $, $ P_4 $ respectively, by $ t_1 $, $ t_2$, $ t_3$, $ t_4$ the strict transforms of a general line through $ P_1 $, $ P_2 $, $ P_3 $, $ P_4 $ respectively, and by $ h_{ij} $ the strict transforms of the line through the two points $ P_i $ and $ P_j $ for all $ i, j \in \left\lbrace 1,2,3,4\right\rbrace  $ respectively. The anti-canonical class 
	\begin{align*}
	-K_{Y} \sim 3l-e_1 - e_2 - e_3 - e_4.
	\end{align*}
	\noindent
	\end{Notation}

		We consider the following divisors of $ Y $:
	\begin{align*}
		B_{1}&:= e_1 + h_{23} + h_{24} + t_{22}  \sim 3l+e_1 - 3e_2 - e_3 - e_4; &L_{1}&: = 3l-2e_1-e_3 - e_4;\\
		B_{2}&:= e_2 + h_{13} + h_{34} + t_{33}  \sim 3l-e_1 +e_2 - 3e_3 - e_4; &L_{2}&: = 3l-e_1 -2e_2 - e_4;\\
		B_{3}&:= e_3 + h_{12} + h_{14} + t_{11}  \sim 3l-3e_1 -e_2 +e_3 - e_4; &L_{3}&: = 3l - e_2 -2e_3 - e_4;		
	\end{align*}
	\noindent
	where $ t_{11} \in \left| t_1\right| $, $ t_{22} \in \left| t_2\right| $ and $ t_{33} \in \left| t_3\right| $ are distinct divisors of $ Y $ such that no more than two of these divisors $ B_{1}, B_{2}, B_{3} $ go through the same point. The branch locus of the divisor $B:=B_1+B_2+B_3$ can be expressed as the following Dynkin diagram: 
   \begin{center}

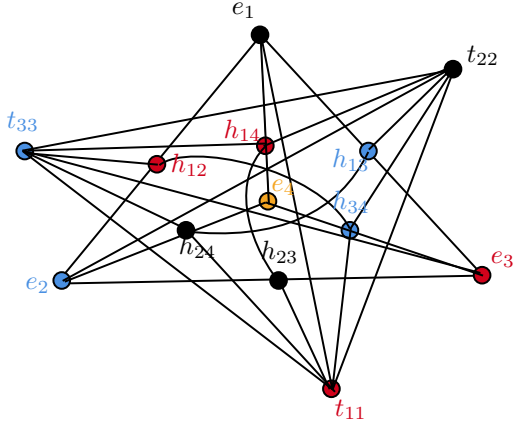
\tikzset{every picture/.style={line width=0.75pt}} 

\begin{tikzpicture}[x=1.0pt,y=1.0pt,yscale=-1,xscale=1]

\draw    (246,138) -- (322.2,45.3) ;
\draw    (246.2,139.3) -- (406.2,136.3) ;
\draw    (322.2,45.3) -- (405.2,135.3) ;
\draw  [fill={rgb, 255:red, 208; green, 2; blue, 27 }  ,fill opacity=1 ] (402.7,136.15) .. controls (402.7,134.41) and (404.11,133) .. (405.85,133) .. controls (407.59,133) and (409,134.41) .. (409,136.15) .. controls (409,137.89) and (407.59,139.3) .. (405.85,139.3) .. controls (404.11,139.3) and (402.7,137.89) .. (402.7,136.15) -- cycle ;
\draw  [fill={rgb, 255:red, 0; green, 0; blue, 0 }  ,fill opacity=1 ] (318.7,45.15) .. controls (318.7,43.41) and (320.11,42) .. (321.85,42) .. controls (323.59,42) and (325,43.41) .. (325,45.15) .. controls (325,46.89) and (323.59,48.3) .. (321.85,48.3) .. controls (320.11,48.3) and (318.7,46.89) .. (318.7,45.15) -- cycle ;
\draw  [fill={rgb, 255:red, 74; green, 144; blue, 226 }  ,fill opacity=1 ] (243.7,138.15) .. controls (243.7,136.41) and (245.11,135) .. (246.85,135) .. controls (248.59,135) and (250,136.41) .. (250,138.15) .. controls (250,139.89) and (248.59,141.3) .. (246.85,141.3) .. controls (245.11,141.3) and (243.7,139.89) .. (243.7,138.15) -- cycle ;
\draw  [fill={rgb, 255:red, 0; green, 0; blue, 0 }  ,fill opacity=1 ] (325.7,138.15) .. controls (325.7,136.41) and (327.11,135) .. (328.85,135) .. controls (330.59,135) and (332,136.41) .. (332,138.15) .. controls (332,139.89) and (330.59,141.3) .. (328.85,141.3) .. controls (327.11,141.3) and (325.7,139.89) .. (325.7,138.15) -- cycle ;
\draw  [fill={rgb, 255:red, 208; green, 2; blue, 27 }  ,fill opacity=1 ] (279.7,94.15) .. controls (279.7,92.41) and (281.11,91) .. (282.85,91) .. controls (284.59,91) and (286,92.41) .. (286,94.15) .. controls (286,95.89) and (284.59,97.3) .. (282.85,97.3) .. controls (281.11,97.3) and (279.7,95.89) .. (279.7,94.15) -- cycle ;
\draw  [fill={rgb, 255:red, 74; green, 144; blue, 226 }  ,fill opacity=1 ] (359.7,89.15) .. controls (359.7,87.41) and (361.11,86) .. (362.85,86) .. controls (364.59,86) and (366,87.41) .. (366,89.15) .. controls (366,90.89) and (364.59,92.3) .. (362.85,92.3) .. controls (361.11,92.3) and (359.7,90.89) .. (359.7,89.15) -- cycle ;
\draw  [fill={rgb, 255:red, 245; green, 166; blue, 35 }  ,fill opacity=1 ] (321.7,108.15) .. controls (321.7,106.41) and (323.11,105) .. (324.85,105) .. controls (326.59,105) and (328,106.41) .. (328,108.15) .. controls (328,109.89) and (326.59,111.3) .. (324.85,111.3) .. controls (323.11,111.3) and (321.7,109.89) .. (321.7,108.15) -- cycle ;
\draw  [fill={rgb, 255:red, 208; green, 2; blue, 27 }  ,fill opacity=1 ] (320.7,87.15) .. controls (320.7,85.41) and (322.11,84) .. (323.85,84) .. controls (325.59,84) and (327,85.41) .. (327,87.15) .. controls (327,88.89) and (325.59,90.3) .. (323.85,90.3) .. controls (322.11,90.3) and (320.7,88.89) .. (320.7,87.15) -- cycle ;
\draw  [fill={rgb, 255:red, 74; green, 144; blue, 226 }  ,fill opacity=1 ] (352.7,119.15) .. controls (352.7,117.41) and (354.11,116) .. (355.85,116) .. controls (357.59,116) and (359,117.41) .. (359,119.15) .. controls (359,120.89) and (357.59,122.3) .. (355.85,122.3) .. controls (354.11,122.3) and (352.7,120.89) .. (352.7,119.15) -- cycle ;
\draw  [fill={rgb, 255:red, 0; green, 0; blue, 0 }  ,fill opacity=1 ] (290.7,119.15) .. controls (290.7,117.41) and (292.11,116) .. (293.85,116) .. controls (295.59,116) and (297,117.41) .. (297,119.15) .. controls (297,120.89) and (295.59,122.3) .. (293.85,122.3) .. controls (292.11,122.3) and (290.7,120.89) .. (290.7,119.15) -- cycle ;
\draw [fill={rgb, 255:red, 0; green, 0; blue, 0 }  ,fill opacity=1 ]   (322,44) -- (325.2,109.3) ;
\draw    (324,108.6) -- (406.2,136.3) ;
\draw    (324,108.6) -- (248,138.6) ;
\draw  [fill={rgb, 255:red, 208; green, 2; blue, 27 }  ,fill opacity=1 ] (345.7,179.15) .. controls (345.7,177.41) and (347.11,176) .. (348.85,176) .. controls (350.59,176) and (352,177.41) .. (352,179.15) .. controls (352,180.89) and (350.59,182.3) .. (348.85,182.3) .. controls (347.11,182.3) and (345.7,180.89) .. (345.7,179.15) -- cycle ;
\draw  [fill={rgb, 255:red, 0; green, 0; blue, 0 }  ,fill opacity=1 ] (391.7,58.15) .. controls (391.7,56.41) and (393.11,55) .. (394.85,55) .. controls (396.59,55) and (398,56.41) .. (398,58.15) .. controls (398,59.89) and (396.59,61.3) .. (394.85,61.3) .. controls (393.11,61.3) and (391.7,59.89) .. (391.7,58.15) -- cycle ;
\draw  [fill={rgb, 255:red, 74; green, 144; blue, 226 }  ,fill opacity=1 ] (229.7,89.15) .. controls (229.7,87.41) and (231.11,86) .. (232.85,86) .. controls (234.59,86) and (236,87.41) .. (236,89.15) .. controls (236,90.89) and (234.59,92.3) .. (232.85,92.3) .. controls (231.11,92.3) and (229.7,90.89) .. (229.7,89.15) -- cycle ;
\draw    (356.2,118.3) -- (349.2,179.3) ;
\draw    (294.2,119.3) -- (350.2,180.3) ;
\draw    (247.2,138.3) -- (395.2,57.3) ;
\draw [color={rgb, 255:red, 0; green, 0; blue, 0 }  ,draw opacity=1 ]   (322.2,45.3) -- (349.2,179.3) ;
\draw    (395.2,58.3) -- (349.2,179.7) ;
\draw    (232,89) -- (395.2,58.3) ;
\draw    (329,138) -- (348.2,179.3) ;
\draw    (232,89) -- (405.2,135.3) ;
\draw    (232.2,89.3) -- (349.2,179.3) ;
\draw    (234,90) -- (283.2,94.3) ;
\draw    (232,89) -- (295.2,119.3) ;
\draw    (234,89) -- (324.2,86.3) ;
\draw    (394.8,57.5) -- (363,88.6) ;
\draw    (395,57) -- (323.2,87.3) ;
\draw    (395,58) -- (356,118.6) ;
\draw    (328.8,138.5) .. controls (315.8,121.5) and (311.8,99.5) .. (323.8,87.5) ;
\draw    (362.8,89.5) .. controls (351.8,112.5) and (328.8,123.5) .. (292.8,119.5) ;
\draw    (355.8,118.5) .. controls (343.8,95.5) and (294.8,85.5) .. (283.8,94.5) ;

\draw (310,32) node [anchor=north west][inner sep=0.75pt]   [align=left] {$e_1$};
\draw (287,89) node [anchor=north west][inner sep=0.75pt]  [color={rgb, 255:red, 208; green, 2; blue, 27 }  ,opacity=1 ] [align=left] {$h_{12}$};
\draw (232,136) node [anchor=north west][inner sep=0.75pt]  [color={rgb, 255:red, 74; green, 144; blue, 226 }  ,opacity=1 ] [align=left] {$e_2$};
\draw (321,122) node [anchor=north west][inner sep=0.75pt]   [align=left] {$h_{23}$};
\draw (408,126) node [anchor=north west][inner sep=0.75pt]  [color={rgb, 255:red, 208; green, 2; blue, 27 }  ,opacity=1 ] [align=left] {$e_3$};
\draw (348,87) node [anchor=north west][inner sep=0.75pt]  [color={rgb, 255:red, 74; green, 144; blue, 226 }  ,opacity=1 ] [align=left] {$h_{13}$};
\draw (325,99) node [anchor=north west][inner sep=0.75pt]  [color={rgb, 255:red, 245; green, 166; blue, 35 }  ,opacity=1 ] [align=left] {$e_4$};
\draw (290,120) node [anchor=north west][inner sep=0.75pt]   [align=left] {$h_{24}$};
\draw (348,103) node [anchor=north west][inner sep=0.75pt]  [color={rgb, 255:red, 74; green, 144; blue, 226 }  ,opacity=1 ] [align=left] {$h_{34}$};
\draw (349,182) node [anchor=north west][inner sep=0.75pt]  [color={rgb, 255:red, 208; green, 2; blue, 27 }  ,opacity=1 ] [align=left] {$t_{11}$};
\draw (399,48) node [anchor=north west][inner sep=0.75pt]   [align=left] {$t_{22}$};
\draw (225,73) node [anchor=north west][inner sep=0.75pt]  [color={rgb, 255:red, 74; green, 144; blue, 226 }  ,opacity=1 ] [align=left] {$t_{33}$};
\draw (307,76) node [anchor=north west][inner sep=0.75pt]  [color={rgb, 255:red, 208; green, 2; blue, 27 }  ,opacity=1 ] [align=left] {$h_{14}$};

\end{tikzpicture}

   \end{center}

	\noindent
	The divisors $ B_{1}, B_{2}, B_{3}, L_{1}, L_{2}, L_{3} $ define a $ \mathbb{Z}/2\mathbb{Z} \times \mathbb{Z}/2\mathbb{Z} $-cover $ \xymatrix{\varphi: X \ar[r] & Y}  $ branched along $B$ (cf. \cite{MR1718139, MR1103912}). The surface $ X $ satisfies the following:		
	\begin{align*}
		2K_{X} &\sim \varphi^*\left( 3l-e_1 - e_2 - e_3 - e_4\right) \sim \varphi^*\left( -K_{Y}\right).
	\end{align*}
	
	 Since the divisor $ 2K_{X} $ is the pull-back of a nef and big divisor, the canonical divisor $ K_{X} $ is nef and big. Thus, the surface $ X $ is of general type and minimal. Furthermore, the surface $X$ possesses the following invariants:
	\begin{align*}
		K_{X}^2=& \left( 3l-e_1 - e_2 - e_3 - e_4\right)^2=5,\\
		p_g\left( X\right) =&\ p_g\left(Y \right) +h^0\left( Y, L_{1} + K_{Y} \right) +h^0\left( Y, L_{2} + K_{Y} \right) +h^0\left( Y, L_{3} + K_{Y} \right)=0,\\
		\chi\left( \mathcal{O}_{X}\right) =&\ 4\chi\left( \mathcal{O}_{Y} \right)  +\frac{1}{2}L_{1}\left( L_{1}+K_{Y}\right) +\frac{1}{2}L_{2}\left( L_{2}+K_{Y}\right) +\frac{1}{2}L_{3}\left( L_{3}+K_{Y}\right)=1,\\
		q\left( X\right)  =&\ 1+p_g\left( X\right)-\chi\left( \mathcal{O}_{X}\right)=0.
	\end{align*}

   From now on, we denote by $ T_{ii}:=\varphi^{*}\left( t_{ii}\right)_{\text{red}}  $, $ E_i:=\varphi^{*}\left( e_i\right)_{\text{red}}  $ for all $ i = 1,2,3 $, $\Tilde{E}_4:=\varphi^{*}(e_4)$, and $ H_{ij}:=\varphi^{*}\left( h_{ij}\right)_{\text{red}}  $, for all $ i \ne j \in \{1,2,3,4\} $. Here are some elementary properties of secondary Burniat surfaces:
 
    \begin{Lemma}
        Let $X$ be a secondary Burniat surface with $ K^2_X = 5 $. Then
        \begin{enumerate}
            \item $K_X\cdot E_i = 1$ for all $i = 1,2,3$, and $K_X\cdot\Tilde{E}_4 = 2$;
            \item $K_X\cdot\varphi^{*}(l) = 6$, and $K_X\cdot\varphi^{*}(t_i) = 4$ for all $i = 1,2,3,4$;
            \item $K_X\cdot H_{ij} = 1$  for all $i,j = 1,2,3,4$ and $i\neq j$;
            \item $K_X\cdot T_{ii} = 2$ for all $i = 1,2,3$.
        \end{enumerate}
    \end{Lemma}

     For a deeper discussion of $ \mathbb{Z}/2\mathbb{Z} \times \mathbb{Z}/2\mathbb{Z} $-covers, we refer the reader to \cite{MR1718139, MR1103912}. Here we recall a fact that we are going to utilize in the present paper.
    \begin{Lemma}(cf. \cite{MR1718139, MR1103912} and \cite[Proposition 1.6]{MR444676})\label{pluricanonical_system_bidouble_covers}
         Let $ \xymatrix{\varphi: X \ar[r] & Y}  $be a smooth $ \mathbb{Z}/2\mathbb{Z} \times \mathbb{Z}/2\mathbb{Z} $-cover with the building data $\{B_1,B_2,B_3,L_1,L_2,L_3\}$ where $X$ is a secondary Burniat surface with $K_X^2=5$ and $Y$ is a del Pezzo surface in Notation 1. For a positive integer $n$ and each $ i = 1,2,3$ with $\{i,j,k\} = \{1,2,3\}$, one has 
        \begin{enumerate}
            \item $\left| 2nK_X\right|_0 = \varphi^{*}\left|n(2K_Y+B)\right|$;
            \item $\left| 2nK_X\right|_i = R_j + R_k + \left|\varphi^{*}(n(2K_Y+B)-L_i)\right|$;
            \item $\left| (2n+1)K_X\right|_0 = R + \varphi^{*}\left|(2n+1)K_Y+nB\right|$;
            \item $\left| (2n+1)K_X\right|_i = R_i + \left|\varphi^{*}((2n+1)K_Y+nB+L_i)\right|$.
        \end{enumerate}
    \end{Lemma}


 \begin{Lemma}\label{comarelct}
   		Suppose $L_1$ and $L_2$ are two distinct lines in $\mathbb{A}^2$. Let $n_1$ and $n_2$ be two integers with $n_2\geq n_1>0$. Then $\lct(\mathbb{A}^2, n_1L_1+n_2L_2)=\frac{1}{n_2}$.
   \end{Lemma}
     Lemma \ref{comarelct} follows from \cite[Example 8.17]{MR1492525} or \cite[Proposition 2.2]{MR1704476}.



\section{Global log canonical threshold of the secondary Burniat surface with $ K^2 = 5 $}
   The aim of this section is to prove Theorem \ref{MainTheorem1}. Since $\glct(X,K_X)=2\glct(X, 2K_X)$, it suffices to prove that $\glct(X, 2K_X) = \frac{1}{4}$. We divide the proof of this result into the following two propositions.
   	\begin{Proposition}\label{upperboundofglct}  $ \glct \left( X,2K_X\right) \le \frac{1}{4} $
   	\end{Proposition}
    
   \begin{proof}
   We shall show that there exists an effective divisor $ E \in \left| 2K_X \right|  $ such that $ \lct\left( X,E\right) = \frac{1}{4} $. Indeed, let $ P\in X $ such that $ \varphi\left( P\right) \in  e_3\cap h_{13}$. We consider the following divisor:
   	\begin{align*}
   		e:=2h_{13} + e_3 + e_1 + h_{24} \sim 3l-e_1 - e_2 - e_3 - e_4.
   	\end{align*}
   \noindent
   The divisor $E:= \varphi^{*}\left( e\right) =4H_{13} + 2E_3 + 2E_1 + 2H_{24} \sim 2K_X$. This implies that $ \lct\left( X,E\right) = \frac{1}{4} $ by Lemma \ref{comarelct}.
   \end{proof}

    \begin{Proposition}\label{lowerboundofglct}
   	$ \glct \left( X,2K_X\right) \ge \frac{1}{4} $.
   \end{Proposition}
	\begin{proof}
	Suppose on the contrary that $ \glct \left( X,2K_X\right) < \frac{1}{4} $. There exists an effective $ \mathbb{Q} $-Cartier divisor $ D \sim_{\mathbb{Q}} 2K_X$ such that $ \left( X, \frac{1}{4}D\right)  $ is not log canonical at some point $ P \in X $. By \cite[Proposition 9.5.11]{MR2095471} ,$$\mult_P\left( D\right) >4.$$

    \noindent We denote by $ d:= \varphi\left( D\right)  $. By \cite[Lemma 4.1]{MR4216579}, $ \left( Y, \frac{1}{4}d +\frac{1}{2}B\right)  $ is not log canonical at $ \varphi\left( P\right)$ where $B:=B_1+B_2+B_3$.\\
	
	\noindent
	\textbf{Case 1}: Assume that $\varphi\left( P\right) \notin B  $. Since $ \left( Y,  \frac{1}{4}d \right)  $ is not log canonical at $ \varphi\left( P\right)  $,  $ \glct\left( Y, d\right) < \frac{1}{4} $. We notice that $ d \sim_\mathbb{Q} -K_{Y} $ and that $ \glct\left( Y, -K_{Y}\right) = \frac{1}{2} $ by \cite[Theorem 1.7]{MR2465686}. We get a contradiction.\\
	
	\noindent
	\textbf{Case 2}: Assume that $\varphi\left( P\right) \in B  $. We shall show that we get a contradiction if $P \in E_3 \cup \left( H_{13}\cup H_{23} \cup H_{34} \right) \cup T_{33} $.  Since $D \sim_{\mathbb{Q}} 2K_X$ and $2K_X \sim \varphi^*\left( 3l-e_1 - e_2 - e_3 - e_4\right)$ is symmetric with respect to $e_1, e_2, e_3, e_4$, the same proof works for $P \in E_1 \cup \left( H_{12}\cup H_{13} \cup H_{14} \right) \cup T_{11} $ and $P \in E_2 \cup \left( H_{12}\cup H_{23} \cup H_{24} \right) \cup T_{22} $.\\

    \noindent
	\textbf{Case 2.1}: Assume that $P \in E_3\setminus H_{13} $. We write
	\begin{align*}
		D = a_3E_3 + a_{13}H_{13} + \Omega,
	\end{align*}
	\noindent
	$ a_3, a_{13} \ge 0 $ and $\Omega$ is an effective $\mathbb{Q}$-divisor satisfying $ E_3,H_{13} \nsubseteq \support\left( \Omega\right) $.
	We have
	\begin{align*}
		8 = D\cdot\varphi^{*}\left( t_3\right) = \left( a_3E_3 + a_{13}H_{13} + \Omega\right)\cdot \varphi^{*}\left( t_3\right) \ge 2a_3.
	\end{align*}
	\noindent
	So $ a_3 \le 4 $. Since $ \left( X, \frac{1}{4}D\right)  $ is not log canonical at $ P $ and $ \frac{a_3}{4} \le 1$, $ \left( X, E_3 +\frac{1}{4}\Omega\right)  $ is not log canonical at $ P $. By the inversion of adjunction formula, $ \left( E_3, \frac{1}{4}\Omega\mid_{E_3}  \right)  $ is not log canonical at $P$. So 
	$$ \mult_P\left( \Omega\mid_{E_3} \right) >4. $$
	\noindent
	This yields that 
	\begin{align*}
		2 + a_3  - a_{13}  &=\left(  D- a_3E_3 - a_{13}H_{13} \right)\cdot E_3 \\ 
        &= \Omega\cdot E_3 \geq \mult_P\left( \Omega\mid_{E_3} \right) >4.
	\end{align*}
	
    \noindent
    On the other hand, we have
	\begin{align*}
		2 = D\cdot H_{13} = \left( a_3E_3 + a_{13}H_{13} + \Omega\right)\cdot H_{13} \ge a_3 - a_{13}.
	\end{align*}
 
	\noindent 
	These above inequalities imply that $ 2 < a_3 - a_{13} \le 2$ which is a contradiction.\\

    \noindent
	\textbf{Case 2.2}: The previous proof works for $P \in H_{13}\setminus E_{3} $, $P \in H_{23}\setminus E_{3} $ and $P \in H_{34}\setminus E_{3} $.\\

	\noindent
	\textbf{Case 2.3}: Suppose that $P \in E_3 \cap H_{13}$.\\
    \noindent
    Write
    $$D = a_3E_3 + a_{13}H_{13} + a_{24}H_{24} + \Omega,$$
    \noindent
    where $a_3, a_{13}, a_{24} \ge 0$ and $E_3, H_{13}, H_{24} \nsubseteq \support(\Omega)$. We have
    \begin{align*}
        8=D\cdot \varphi^{*}(t_2)= (a_3E_3 + a_{13}H_{13} + a_{24}H_{24} + \Omega)\cdot\varphi^{*}(t_2) \ge2a_{13}.
    \end{align*}
    \noindent
    So $ a_{13} \le 4 $. Since $ \left( X, \frac{1}{4}D\right)  $ is not log canonical at $ P $ and $ \frac{a_{13}}{4} \le 1$, $ \left( X, H_{13} +\frac{1}{4}(a_3E_3 +\Omega)\right)  $ is not log canonical at $ P $. By the inversion of adjunction formula $ \left( H_{13}, \frac{1}{4}(a_3E_3 +\Omega)\mid_{H_{13}}  \right)  $ is not log canonical at $ P $. Thus
    \begin{align*}
    2+a_{13}-a_{24}&=H_{13}\cdot (D-a_{13}H_{13}-a_{24}H_{24})=H_{13}\cdot(a_3E_3+\Omega)\\
&\geq \mult_P\left( H_{13} \right)\mult_P\left( a_3E_3+\Omega \right) >4
    \end{align*}

    \noindent
    On the other hand, one has 
    \begin{align*}
        0\leq \Omega \cdot H_{24} = (D - a_3E_3 - a_{13}H_{13} -a_{24}H_{24})\cdot H_{24} = 2 - a_{13} + a_{24}.
    \end{align*}
    The two above inequalities imply that $2 < a_{13}-a_{24} \le 2$ which is a contradiction. \\

	\noindent
	\textbf{Case 2.4}: Assume that $P \in T_{33}  $. Write
	\begin{align*}
		D = a_{33}T_{33} + \Omega,
	\end{align*}
	\noindent
	$ a_{33} \ge 0 $ and $ T_{33} \nsubseteq \support\left( \Omega\right) $. We have
	\begin{align*}
		8 = D\cdot\varphi^{*}\left( t_2\right) \ge 2a_{33}.
	\end{align*}
	\noindent
	So $ a_{33} \le 4 $. Since $ \left( X, \frac{1}{4}D\right)  $ is not log canonical at $ P $ and $ \frac{a_{33}}{4} \le 1$,  $ \left( X, T_{33} +\frac{1}{4}\Omega\right)  $  is not log canonical at $ P $. By the inversion of adjunction formula $ \left( T_{33}, \frac{1}{4}\Omega\mid_{T_{33}}  \right)  $ is not log canonical at $ P $. So 
	$$ T_{33}\cdot\Omega \ge \mult_P\left( \Omega\mid_{T_{33}} \right) >4. $$
	
	\noindent
	However, $ 4 = T_{33}\cdot \left( D - a_{33}T_{33}\right) = T_{33}\cdot \Omega $ which is a contradiction. We complete the proof of Proposition \ref{lowerboundofglct}.
\end{proof}

\section{Even pluricanonical linear system}
This section is devoted to prove Theorem \ref{MainTheorem-2}. It is clear that for any $D_0\in |2nK_X|_0$, $\lct\left( X,D_{0}\right) \ge \frac{1}{4n}$ is an immediate consequence of Theorem \ref{MainTheorem1}. The proof is completed by showing that $\lct\left( X,D_{i}\right) \ge \frac{1}{4n-3}$ for any $D_{i} \in \left| 2nK_X\right|_i  $ and all $i = 1,2,3$. Without loss of generality we may assume that $i=1$ by Lemma \ref{pluricanonical_system_bidouble_covers}.\\

\noindent
 We first note that $\lct_P(X,\overline{D}_1) = \frac{1}{4n-3}$ by Lemma \ref{comarelct} where $P \in H_{12}\cap H_{34}$ and   
    $$\overline{D}_1 = R_2 + R_3 + (4n-4)H_{12} + (2n-2)H_{34} + 2nE_1 + (2n-4)E_2 \in \left| 2nK_X\right|_1. $$

\noindent
    Suppose on the contrary that $\lct\left( X,D_{1}\right) < \frac{1}{4n-3}$ for some member $D_1 \in \left| 2nK_X\right|_1$. Then    $(X,\frac{1}{4n-3}D_1)$ is not log canonical at some point $P$. We have $\mult_P(D_1) > 4n-3$. Let us denote by $d_1:= \varphi(D_1)$. By \cite[Lemma 4.1]{MR4216579}, $(Y, \frac{1}{4n-3}d_1 + \frac{1}{2}B)$ is not log canonical at $\varphi(P)$. Assume that $\varphi(P) \notin B$. Then $(Y, \frac{1}{4n-3}d_1)$ is not log canonical at $\varphi(P)$. So $\glct(Y, d_1) < \frac{1}{4n-3}$. We notice that $d_1 \sim_{\mathbb{Q}} n(-K_{Y})$ and $\glct(Y, -K_{Y}) \ge \frac{1}{2}$ by \cite[Theorem 1.7]{MR2465686}. Thus $\glct(Y, d_1) \ge \frac{1}{2n}$ which is a contradiction. Therefore $\varphi(P) \in B=B_1 + B_2 + B_3$.\\

\noindent
By Lemma \ref{pluricanonical_system_bidouble_covers},
        \begin{align*}
        \left| 2nK_X\right|_1 = &(E_2+H_{13}+H_{34}+T_{33})+(E_3+H_{12}+H_{14}+T_{11})  \\ &+\left| \varphi^{*}((3n-3)l-(n-2)e_1-ne_2-(n-1)e_3-(n-1)e_4)\right|.\end{align*}

    \noindent

     \noindent It remains to show that we get a contradiction if $\varphi(P) \in B=B_1 + B_2 + B_3$.\\

 \noindent
    {\bf Case 1:} $P\in E_1 \cap H_{12} $.\\
    Write
    $$D_1 = a_1E_1 + a_2E_2+ a_{12}H_{12} + a_{34}H_{34} + \Omega,$$
    \noindent
    where $a_1 \ge 0$, $a_2, a_{12}, a_{34} \ge 1$ and $E_1, E_2, H_{12}, H_{34} \nsubseteq \support(\Omega)$. 
    Now 
    \begin{align*}&\Omega\cdot \frac{1}{2}\varphi^*(t_3)\geq (H_{13}+T_{33}+E_3+H_{14}+T_{11})\cdot \frac{1}{2}\varphi^*(t_3)=3.
    \end{align*}
    Thus $4n=D_1\cdot T_{33}= a_{12}+\Omega\cdot T_{33}\geq a_{12}+3$ which means $a_{12}\leq 4n-3$.

    \noindent
    So $(X,H_{12} + \frac{1}{4n-3}(a_1E_1+\Omega))$ is not log canonical at the point $P$. The inversion of adjunction formula gives that the pair $(H_{12},\frac{1}{4n-3}(a_1E_1+\Omega)|_{H_{12}})$ is not log canonical at $P$. Thus
    \begin{align*}
    &2n-a_2+a_{12}-a_{34}=(D_1- a_2E_2- a_{12}H_{12} - a_{34}H_{34})\cdot H_{12}\\
    &\geq 
    \mult_P(H_{12})\mult_P(a_1E_1+\Omega)> 4n-3.
    \end{align*}
    As $a_2\geq 1$, we have $a_{12}-a_{34}> 2n-2$.
    On the other hand, one has 
    \begin{align*}
    &\Omega\cdot H_{34}\geq (H_{13}+T_{33}+E_3+H_{14}+T_{11})\cdot H_{34}=2
    \end{align*}
    which implies \[2n=D_1\cdot H_{34}=a_{12}-a_{34}+\Omega\cdot H_{34}\geq a_{12}-a_{34}+2>2n-2+2=2n.\] 
    It is a contradiction.\\

    \noindent
    {\bf Case 2:} $P\in E_1 \setminus H_{12} $.\\
    Write
    $$D_1 = a_1E_1 + a_{12}H_{12} +  a_{34}H_{34}+ \Omega,$$
    \noindent
    where $a_1 \ge 0, a_{12}, a_{34} \ge 1$ and $E_1, H_{12}, H_{34} \nsubseteq \support(\Omega)$. We notice that $\Omega \ge E_2 + H_{13} + T_{33} + E_3 + H_{14}+ T_{11}$.
    \noindent
    Then
    \begin{align*}
        2n &= D_1\cdot H_{12} = a_1 - a_{12} + a_{34}+\Omega \cdot H_{12}\ge a_{1} -a_{12} +a_{34}+ 2 \textup{ and }\\
        2n &= D_1\cdot H_{34} =a_{12} -a_{34}+\Omega \cdot H_{34}\ge   a_{12} -a_{34}+2
    \end{align*}
    which yield $$a_1\leq 4n-4<4n-3.$$

    \noindent
    So $(X,E_1 + \frac{1}{4n-3}\Omega)$ is not log canonical at the point $P$. The inversion of adjunction formula implies that the pair $(E_1,\frac{1}{4n-3}\Omega|_{E_1})$ is not log canonical at $P$. Thus $$\mult_P(E_1)\mult_P(\Omega) > 4n-3.$$

    \noindent
    This yields that
    $$2n +a_1 -a_{12} = E_1\cdot(D_1 - a_1E_1 -a_{12}H_{12} -a_{34}H_{34}) = E_1\cdot \Omega > 4n-3.$$
    \noindent
    This inequality together with the above one 
    $$2n \ge a_1 - a_{12} + a_{34}+2 \geq a_1 - a_{12} +3$$
    \noindent
    implies that $2n \ge a_1 - a_{12} +3 > 2n $ which is a contradiction.\\

    \noindent
    {\bf Case 3:} $P\in H_{14} \setminus E_1 $.\\
    Write 
    $$D_1 = R_2 + R_3 + D'_1,$$
    \noindent
    where $D'_1 \in \left| \varphi^{*}((3n-3)l-(n-2)e_1-ne_2-(n-1)e_3-(n-1)e_4)\right|.$ Since $(X, \frac{1}{4n-3}D_1)$ is not log canonical at $P$, $(X, \frac{1}{4n-3}(H_{14} + T_{33} +D'_1))$ is not log canonical at $P$. Write
    $$D'_1 = a_{14}H_{14} + a_1E_1 + \Omega,$$
    \noindent
    where $a_{14}, a_1 \ge 0$ and $H_{14},E_1 \nsubseteq \support(\Omega)$. We have
    $$8n - 8 = D'_1\cdot\varphi^*(t_3) \ge 2a_{14}.$$
    \noindent
    So $a_{14}+1 \le 4n - 3 $. This implies that $(X, H_{14} + \frac{1}{4n-3}(T_{33} + \Omega))$ is not log canonical at $P$. We obtain that
    $$\mult_P(H_{14})\mult_P(T_{33} + \Omega) > 4n-3.$$
    \noindent
    Thus
    \begin{align*}
        1 + 2n +a_{14} -a_1 &\ge \mult_P(H_{14})\mult_P(T_{33} + D'_1 - a_{14}H_{14} - a_1E_1)\\
        &=\mult_P(H_{14})\mult_P(T_{33} + \Omega)\\
        &> 4n -3.
    \end{align*}
    \noindent
    Therefore $a_{14} -a_1 > 2n -4$.\\

    \noindent
    On the other hand, one has
    $$0 \le E_1\cdot\Omega = E_1\cdot(D'_1 - a_{14}H_{14} - a_1E_1 ) = 2n - 4 - a_{14} + a_1.$$
    \noindent
    Thus $a_{14}-a_1 \le 2n -4$. This leads to the fact that $$a_{14} -a_1 > 2n -4\ge a_{14}-a_1$$ which is a contradiction.\\

    \noindent
    {\bf Case 4:}  $P\in E_2\cap H_{23}$.\\
    Write
    $$D_1 = a_2E_2 + a_{23}H_{23}+a_{24}H_{24} + \Omega,$$
    \noindent
    where $a_2 \ge 1, a_{23}, a_{24} \ge 0$ and $E_2, H_{23}, H_{24} \nsubseteq \textup{supp}(\Omega)$. 
    Now 
    \begin{align*}&\Omega\cdot \frac{1}{2}\varphi^*(t_2)\geq (R_2+R_3-E_2)\cdot \frac{1}{2}\varphi^*(t_2)\\
    &=(H_{13}+H_{34}+T_{33}+E_3+H_{12}+H_{14}+T_{11})\cdot T_{22}=5.
    \end{align*}
    Thus, $4n=D_1\cdot T_{22}= a_{2}+\Omega\cdot T_{22}\geq a_{2}+5$. So $a_{2}\leq 4n-5< 4n-3$.

    \noindent
    We get that $(X,E_2 + \frac{1}{4n-3}(a_{23}H_{23}+\Omega))$ is not log canonical at the point $P$. The inversion of adjunction formula yields that the pair $(E_2,\frac{1}{4n-3}(a_{23}H_{23}+\Omega)|_{E_{2}})$ is not log canonical at $P$. Thus 
    \begin{align*}
    &2n+a_2-a_{24}=(D_1- a_{2}E_2 - a_{24}H_{24})\cdot E_{2}\\
    &\geq 
    \mult_P(E_2)\mult_P(a_{23}H_{23}+\Omega)\\
    &> 4n-3.
    \end{align*}
    We have $a_{2}-a_{24}> 2n-3$.
    On the other hand, one sees  
    \begin{align*}
    &\Omega\cdot H_{24}\geq (R_2+R_3-E_2)\cdot H_{24}\\
    &=(H_{13}+H_{34}+T_{33}+E_3+H_{12}+H_{14}+T_{11})\cdot H_{24}=3,
    \end{align*}
    which implies \[2n=D_1\cdot H_{24}=a_{2}-a_{24}+\Omega\cdot H_{24}\geq a_{2}-a_{24}+3>2n-3+3=2n.\] 
    It is a contradiction.\\

    \noindent
    {\bf Case 5:}  $P\in E_2\setminus  H_{23}$.\\
    We have a contradiction similar to Case 2. Indeed, 
    write
    $$D_1 = a_2E_2  + a_{14}H_{14} + a_{23}H_{23} + \Omega,$$
    \noindent
    where $a_{2}, a_{14} \ge 1$, $a_{23}\ge 0$ and $E_2, H_{14}, H_{23} \nsubseteq \support(\Omega)$.\\ 
    \noindent
    We notice that $\Omega \ge H_{34} +T_{33} +E_{3} + H_{12}+ H_{13} + T_{11}$. Then 
    \begin{align*}
        2n &= D_1\cdot H_{14} = -a_{14} +a_{23} + \Omega \cdot H_{14}\ge -a_{14} +a_{23} +1, \textup{ and }\\
        2n &= D_1\cdot H_{23} = a_2 +a_{14} -a_{23} + \Omega \cdot H_{23}\ge a_2 +a_{14} -a_{23} +2
    \end{align*}
    \noindent
    which yield that
    $$ a_2 \le 4n-3.$$

    \noindent
    So we get $(X,E_2 + \frac{1}{4n-3}\Omega)$ is not log canonical at the point $P$. The inversion of adjunction formula yields that the pair $(E_2,\frac{1}{4n-3}\Omega|_{E_2})$ is not log canonical at $P$. Thus
    $$\mult_P(E_2)\mult_P(\Omega) > 4n-3.$$

    \noindent
    We notice that
    $$E_2\cdot \Omega = E_2\cdot(D_1 - a_2E_2 -a_{14}H_{14} - a_{23}H_{23}) = 2n +a_2 -a_{23}.$$
    So 
    $$2n +a_2 - a_{23} > 4n-3.$$
    \noindent
    This inequality together with the above  inequality 
    $$2n \ge a_2+ a_{14} - a_{23}  +2 \ge a_2 - a_{23} + 3$$
    \noindent
    leads to a contradiction that $2n\geq a_2-a_{23}+3>2n.$\\

    \noindent
    {\bf Case 6:} $P\in (H_{12} \setminus (E_1 \cup E_2)) \cup (H_{13} \setminus (E_1 \cup E_3))$.\\
    Suppose that $P\in H_{12} \setminus (E_1 \cup E_2)$. Write 
    $$D_1 = R_2 + R_3 + D'_1,$$

    \noindent
    where $D'_1 \in \left| \varphi^{*}((3n-3)l-(n-2)e_1-ne_2-(n-1)e_3-(n-1)e_4)\right|.$  Since $(X, \frac{1}{4n-3}D_1)$ is not log canonical at $P$, $(X, \frac{1}{4n-3}(H_{34} + T_{33} + H_{12} +D'_1))$ is not log canonical at $P$. Write
    $$D'_1 = a_{12}H_{12} + a_1E_1 + a_2E_2 + \Omega,$$
    \noindent
    where $a_{12}, a_1, a_2 \ge 0$ and $H_{12},E_1,E_2 \nsubseteq \support(\Omega)$. We have
    $$8n - 8 = D'_1\cdot\varphi^*(t_3) \ge 2a_{12}.$$
    \noindent
    So $a_{12}+1 \le 4n - 3 $. This implies that $(X, H_{12} + \frac{1}{4n-3}(H_{34} + T_{33} + \Omega))$ is not log canonical at $P$. Then
    $$\mult_P(H_{12})\mult_P(H_{34} + T_{33} + \Omega) > 4n-3.$$
    \noindent
    Thus
    \begin{align*}
        1 + 2n - 2 +a_{12} -a_1 - a_2 &\ge \mult_P(H_{12})\mult_P(H_{34} + T_{33} + D'_1 - a_{12}H_{12} - a_1E_1 - a_2E_2)\\
        &=\mult_P(H_{12})\mult_P(H_{34} + T_{33} + \Omega)\\
        &> 4n -3.
    \end{align*}
    \noindent
    Therefore $a_{12} > 2n -2+a_1 +a_2$.\\

    \noindent
    On the other hand, one has 
    $$0 \le E_1\cdot\Omega = E_1\cdot (D'_1 - a_{12}H_{12} - a_1E_1 - a_2E_2) = 2n - 4 - a_{12} + a_1.$$
    \noindent
    Thus $a_{12} \le 2n -4+a_1$. This leads to the fact that $a_2 < -2$ which is a contradiction.\\

    \noindent
    Similarly, we get a contradiction if $P\in H_{13} \setminus (E_1 \cup E_3)$.\\

    \noindent
    {\bf Case 7:} $P\in E_3\cap H_{23}$.  \\
 We shall obtain a contradiction analogous to Cases 3 and 6. 

\noindent
Write 
$$D_1 = R_2 + R_3 + D'_1,$$
    \noindent
   where $D'_1 \in \left| \varphi^{*}((3n-3)l-(n-2)e_1-ne_2-(n-1)e_3-(n-1)e_4)\right|.$  Since $(X, \frac{1}{4n-3}D_1)$ is not log canonical at $P$, $(X, \frac{1}{4n-3}(E_3+D'_1))$ is not log canonical at $P$. Write
$$D'_1 = a_{23}H_{23} + a_2E_2 + \Omega,$$
    \noindent
    where $a_{23}, a_2 \ge 0$ and $H_{23},E_2 \nsubseteq \support(\Omega)$. We have
    $$8n - 8 = D'_1\cdot\varphi^*(t_4) \ge 2a_{23}.$$
    \noindent
    So $a_{23} \le 4n - 4 <4n-3 $. This implies that $(X, H_{23} + \frac{1}{4n-3}(E_3+ \Omega))$ is not log canonical at $P$. Then 
    $$\mult_P(H_{23})\mult_P(E_3 + \Omega) > 4n-3.$$
    \noindent
    Thus
    \begin{align*}
        1 + 2n-4 +a_{23} -a_2 &\ge 1+\mult_P(H_{23})\mult_P( D'_1 - a_{23}H_{23} - a_2E_2)\\
        &= 1+\mult_P(H_{23})\mult_P( \Omega)\\
        &\geq \mult_P(H_{23})\mult_P(E_3 + \Omega)
        > 4n -3.
    \end{align*}
    \noindent
    Therefore, $a_{23} -a_2 > 2n$.\\

    \noindent
    On the other hand, we have
    $$0 \le E_2\cdot\Omega = E_2\cdot(D'_1 - a_{23}H_{23} - a_2E_2 ) = 2n - a_{23} + a_2.$$
    \noindent 
    This leads to a contradiction that $$a_{23} -a_2 > 2n\ge a_{23}-a_2.$$

    \noindent
    {\bf Case 8:} $P\in E_3 \setminus H_{23}$.\\
We get a contradiction similar to Case 5.\\

    \noindent
    {\bf Case 9:} $P\in (H_{23} \setminus E_2) \cup (H_{24} \setminus E_2 )$.\\
    Suppose that $P\in H_{23} \setminus E_2$. Write 
    $$D_1 = R_2 + R_3 + D'_1,$$

    \noindent
    where $D'_1 \in \left| \varphi^{*}((3n-3)l-(n-2)e_1-ne_2-(n-1)e_3-(n-1)e_4)\right|.$  Since $(X, \frac{1}{4n-3}D_1)$ is not log canonical at $P$, $(X, \frac{1}{4n-3}(E_3  + H_{14} +T_{11}+D'_1))$ is not log canonical at $P$. Write
    $$D'_1 = a_{23}H_{23} + a_2E_2  + \Omega,$$
    \noindent
    where $a_{23}, a_2 \ge 0$ and $H_{23},E_2 \nsubseteq \support(\Omega)$. We have
    $$8n - 8 = D'_1\cdot\varphi^*(t_4) \ge 2a_{23}.$$
    \noindent
    So $a_{23} \le 4n - 4<4n-3 $. This implies that $(X, H_{23} + \frac{1}{4n-3}(E_3+H_{14} + T_{11} + \Omega))$ is not log canonical at $P$.
    \noindent
    Thus
    \begin{align*}
        1 + 2n - 4 +a_{23} -a_2 &\ge \mult_P(H_{23})\mult_P(E_3+H_{14} + T_{11} + D'_1 - a_{23}H_{23} - a_2E_2 )\\
        &=\mult_P(H_{23})\mult_P(E_3+H_{14} + T_{11} + \Omega)\\
        &> 4n -3.
    \end{align*}
    \noindent
    Therefore $a_{23}-a_2>2n$.\\

    \noindent
    On the other hand, one has 
    $$0 \le E_2\cdot\Omega = E_2\cdot (D'_1 - a_{23}H_{23} - a_2E_2) = 2n - a_{23} + a_2.$$
    \noindent
    It induces a contradiction that 
    $$ a_{23}-a_2 > 2n \ge a_{23} -a_2.$$

    \noindent
    Similarly, we get a contradiction if $P\in H_{24}\setminus E_2$.\\

    \noindent
    {\bf Case 10:} $P\in H_{34} \setminus E_3 $.\\
    Analogously to Case 3, we get a contradiction.\\

    \noindent
    {\bf Case 11:} $P\in T_{22}  $.\\
    \noindent
    Write 
    $$D_1 = R_2 + R_3 + D'_1,$$
    \noindent
    where $D'_1 \in \left| \varphi^{*}((3n-3)l-(n-2)e_1-ne_2-(n-1)e_3-(n-1)e_4)\right|.$  Since $(X, \frac{1}{4n-3}D_1)$ is not log canonical at $P$, $(X, \frac{1}{4n-3}(E_2 + H_{13}+H_{34}+ T_{33} + H_{14}+T_{11}+D'_1))$ is not log canonical at $P$. Write
    $$D'_1 = a_{22}T_{22}+ \Omega,$$
    \noindent
    where $a_{22}\ge 0$ and $T_{22} \nsubseteq \support(\Omega)$. Then 
    $$8n - 8 = D'_1\cdot\varphi^*(t_3) \ge 2a_{22}.$$
    \noindent
    This implies that $(X, T_{22} + \frac{1}{4n-3}(E_2 + H_{13}+H_{34}+ T_{33} + H_{14}+T_{11}+\Omega))$ is not log canonical at $P$. 
    \noindent
    So
    \begin{align*}
        &4n-5\ge 1+\mult_P(T_{22})\mult_P( D'_1 ) \\ &\ge \mult_P(T_{22})\mult_P(E_2 + H_{13}+H_{34}+ T_{33} + H_{14}+T_{11}+ D'_1 - a_{22}T_{22}) \\
        &=\mult_P(T_{22})\mult_P(E_2 + H_{13}+H_{34}+ T_{33} + H_{14}+T_{11}+\Omega)\\
        &> 4n -3
    \end{align*}
    \noindent
    which is a contradiction.\\

    \noindent
    {\bf Case 12:} $P\in T_{33}  $.\\
    \noindent
    We obtain a contradiction similar to Case 11.\\
   
    \noindent
    {\bf Case 13:} $P \in T_{11}$.

    \noindent By previous cases, it is enough to assume that $P\in T_{11}\setminus (E_1\cup H_{23}\cup H_{24}\cup H_{34}\cup T_{22}\cup T_{33})$. 
    Let $d_1:=\varphi(D_1).$ 
    Write 
    $$d_1 = a_{2}'e_2+a_{13}'h_{13}+a_{14}'h_{14}+a_{34}'h_{34}+a_{11}'t_{11}+a_{33}'t_{33} + \Omega',$$ where $a_{2}',a_{13}',a_{14}',a_{34}',a_{11}',a_{33}'\ge 0$ and $\Omega'$ is an effective divisor on $Y$ such that  $e_{2},h_{13}, h_{14},h_{34},t_{11},t_{33}\nsubseteq \support(\Omega')$. Since $\varphi^*(d_1)=\varphi^*(\varphi(D_1))\geq D_1$ and 
    $$D_1\geq R_2+R_3= E_2+H_{13}+H_{34}+T_{33}+E_{3}+H_{12}+H_{14}+T_{11}$$ where $\varphi^*(e_2)=2E_{2},  \varphi^*(h_{13})=2H_{13}, \varphi^*(h_{14})=2H_{14}, \varphi^*(h_{34})=2H_{34}, \varphi^*(t_{33})=2T_{33}$ (cf. \cite[Section 2]{MR0755236}, \cite[Section 3.2]{MR1802792}, \cite{MR1103912}), one has $\min\{a_{2}',a_{13}',a_{14}',a_{34}',a_{33}'\}\geq \frac{1}{2}$.  Thus $$d_1\cdot t_2\geq (a_{2}'e_2+a_{13}'h_{13}+a_{14}'h_{14}+a_{34}'h_{34}+a_{11}'t_{11}+a_{33}'t_{33})\cdot t_2\ge  a_{11}'+\frac{5}{2}.$$ As  $d_1\sim_{\mathbb{Q}} n(-K_{Y})$, we see 
    $$d_1\cdot t_2=n(3l-e_1-e_2-e_3-e_4)\cdot t_2=2n,$$ which yields $a_{11}'\leq 2n-\frac{5}{2}$.

    Suppose now that $(X,\frac{1}{4n-3}D_1)$ is not log canonical at $P$. Then  $(Y, \frac{1}{4n-3}d_1 + \frac{1}{2}B)$ is not log canonical at $\varphi(P)$ by \cite[Lemma 4.1]{MR4216579}. The assumption $P\in T_{11}\setminus (E_1\cup H_{23}\cup H_{24}\cup H_{34}\cup T_{22}\cup T_{33})$ implies that $(Y, \frac{1}{4n-3}d_1 + \frac{1}{2}t_{11})$ is not log canonical at $\varphi(P)$.  
    Since $a_{11}'\leq 2n-\frac{5}{2}$ and $\varphi(P)\not\in e_2\cup h_{13}\cup h_{14}\cup h_{34}\cup t_{33}$, we have $$\frac{a_{11}'}{4n-3}+\frac{1}{2}<1$$ and $(Y, t_{11}+\frac{1}{4n-3}\Omega')$ is not log canonical at $\varphi(P)$. So 
    $$\mult_{\varphi(P)}(t_{11})\mult_{\varphi(P)}(\Omega') > 4n-3.$$ 
    Hence 
    \begin{align*}
    2n&=nt_1\cdot(3l-e_1 - e_2 - e_3 - e_4) = t_{11}\cdot d_1\\ 
    &\geq t_{11}\cdot \Omega'
    \geq \mult_{\varphi(P)}(t_{11})\mult_{\varphi(P)}(\Omega') > 4n-3,
    \end{align*}
    which contradicts to the assumption $n\ge 2$.\\

\section{Odd pluricanonical linear system}
In this section we prove Theorem \ref{MainTheorem-3}.
    \subsection{The invariant part.} 
        In this part we show that $\lct(X,D_0)\geq \frac{1}{4n-3}$ for any $D_0\in |(2n+1)K_X|_0$. 
         
         Denote by 
    $$\overline{D}_0:=R+(4n-4)H_{13}+(2n-2)E_3+(2n-2)E_1+(2n-2)H_{24}\in |(2n+1)K_X|_0.$$

    \noindent
    By Lemma \ref{comarelct}, $\lct_P(X,\overline{D}_0) = \frac{1}{4n-3}$ where $P \in H_{13}\cap E_{3}$.

        Suppose on the contrary that there is a member $D_0\in |(2n+1)K_X|_0$ such that $(X,\frac{1}{4n-3}D_0)$ is not log canonical at some point $P$. 
By Lemma \ref{pluricanonical_system_bidouble_covers} \[|(2n+1)K_X|_0=R+\varphi^*|-(n-1)K_Y|,\] and the global log canonical threshold of the pair $(X,2(n-1)K_X)$ is $1/(4n-4)$ by Theorem \ref{MainTheorem1}. One sees that the point $P$ is contained in $R$. Write 
\[D_0=R+D'\]
for some member $D'\in |2(n-1)K_X|$.\\

We shall show that we get a contradiction if $P \in E_3 \cup (H_{13} \cup H_{23} \cup H_{34}) \cup  T_{33}$. The same proof works for $P \in E_1 \cup (H_{12} \cup H_{13} \cup H_{14}) \cup T_{11}$ and $P \in E_2 \cup (H_{12} \cup H_{23} \cup H_{24}) \cup T_{22}$.\\

\noindent
{\bf Case 1:} $ P\in E_3\cap H_{13}$.\\
     Since $(X,\frac{1}{4n-3}D_0)$ is not log canonical at $P$, the pair $(X,\frac{1}{4n-3}(E_3+H_{13}+D'))$ is not log canonical at $P$.
    Write
    $$D' = a_3E_3 + a_{13}H_{13} + a_{24}H_{24} + \Omega,$$
    \noindent
    where $a_3, a_{13}, a_{24} \ge 0$ and $E_3, H_{13}, H_{24} \nsubseteq \support(\Omega)$. 
    \begin{align*}
        8n-8=D'\cdot \varphi^{*}(t_2) = (a_3E_3 + a_{13}H_{13} + a_{24}H_{24} + \Omega)\cdot\varphi^{*}(t_2) \ge2a_{13}.
    \end{align*}
    \noindent
    So $ a_{13}+1 \le 4n-3 $. Thus $ \left( X, H_{13} +\frac{1}{4n-3}((a_3+1)E_3 +\Omega)\right)  $ is not log canonical at $ P $. By the inversion of adjunction formula, $ \left( H_{13}, \frac{1}{4n-3}((a_3+1)E_3 +\Omega)\mid_{H_{13}}  \right)  $ is not log canonical at $ P $. Then 
    \begin{align*}
    1+2n-2+a_{13}-a_{24}&=H_{13}\cdot (E_3 + D'-a_{13}H_{13}-a_{24}H_{24})=H_{13}\cdot((a_3+1)E_3+\Omega)\\
&\geq \mult_P\left( H_{13} \right)\mult_P\left( (a_3+1)E_3+\Omega \right) >4n-3.
    \end{align*}

    \noindent
    On the other hand, one has 
    \begin{align*}
        0\leq \Omega \cdot H_{24} = (D' - a_3E_3 - a_{13}H_{13} -a_{24}H_{24})\cdot H_{24} = 2n-2 - a_{13} + a_{24}.
    \end{align*}
    The two above inequalities imply that $2n-2 < a_{13}-a_{24} \le 2n-2$ which is a contradiction.\\

\noindent 
{\bf Case 2:} $P\in E_3\setminus H_{13}.$ 
\\
Since $(X,\frac{1}{4n-3}D_{0})$ is not log canonical at $P$, $(X,\frac{1}{4n-3}(E_3 + H_{23} + H_{34} +T_{33} + D^{'})$ is not log canonical at $P$. Write
$$D^{'} = a_3E_3 + a_{13}H_{13}+ a_{24}H_{24} +\Omega,$$
\noindent
where $a_3, a_{13}, a_{24} \ge 0$ and $E_3, H_{13}, H_{24} \nsubseteq \support(\Omega)$. We have
\begin{align*}
        2n -2&= D'\cdot H_{13} = a_3 - a_{13} + a_{24}+\Omega \cdot H_{13}\ge a_{3} -a_{13} +a_{24} \textup{ and }\\
        2n -2&= D'\cdot H_{24} =a_{13} -a_{24}+\Omega \cdot H_{24}\ge   a_{13} -a_{24}
    \end{align*}
    which yield $$a_3+1\leq 4n-3.$$
\noindent
\noindent
    So $(X,E_3 + \frac{1}{4n-3}(H_{23} + H_{34} +T_{33} +\Omega))$ is not log canonical at the point $P$. The inversion of adjunction formula implies that the pair $(E_3,\frac{1}{4n-3}(H_{23} + H_{34} +T_{33} +\Omega)|_{E_3})$ is not log canonical at $P$. Thus $$1+2n -2 +a_3 -a_{13} \ge \mult_P(E_3)\mult_P(H_{23} + H_{34} +T_{33} +\Omega) > 4n-3.$$
    
    \noindent
    This inequality together with the above one 
    $$2n -2\ge a_3-a_{13} + a_{24} \geq a_3 - a_{13}$$
    \noindent
    implies that $2n -2 \ge a_3 - a_{13}  > 2n -2$ which is a contradiction.\\

\noindent 
{\bf Case 3:} $P\in (H_{13}\setminus  E_{3}) \cup (H_{23}\setminus  E_{3}) \cup (H_{34}\setminus  E_{3}).$ 
\\
Assume that $P\in H_{13}\setminus  E_{3}$. Since $(X,\frac{1}{4n-3}D_{0})$ is not log canonical at $P$, $(X,\frac{1}{4n-3}(E_1+H_{13} +H_{24} +T_{22} + D^{'})$ is not log canonical at $P$. Write
$$D^{'} = a_3E_3+a_{13}H_{13}   +\Omega,$$
\noindent
where $ a_3, a_{13} \ge 0$ and $ E_3, H_{13} \nsubseteq \support(\Omega)$. We have
$$8n - 8 = D^{'}\cdot \varphi^*(t_2) \ge 2a_{13}.$$
\noindent
So $1+a_{13} \le 4n-3.$ This implies that $(X,H_{13} + \frac{1}{4n-3}(E_1+H_{24} +T_{22}+ \Omega))$ is not log canonical at $P$. By the inversion of adjunction formula, $(H_{13}, \frac{1}{4n-3}(E_1+H_{24}+T_{22}+\Omega)|_{H_{13}})$ is not log canonical at $P$. Thus 
$$1+(2n-2) + a_{13} - a_{3} \ge \textup{mult}_P(H_{13})\textup{mult}_P(E_1+H_{24} +T_{22}+ \Omega) >4n-3.$$
\noindent
On the other hand, one has 
$$0 \le \Omega\cdot E_{3} =(D^{'} - a_{13}H_{13} - a_3E_3)\cdot E_{3}= 2n-2 - a_{13} +a_3.$$
\noindent
The two above inequalities yield a contradiction that $$ 2n-2\ge a_{13}-a_3>2n-2.$$

Similarly we get a contradiction if $P\in H_{23}\setminus E_3$ or $P\in H_{34}\setminus E_3$.\\

\noindent 
{\bf Case 4:}  $P\in T_{33}\setminus E_3$
\\
Denote by $E$ the effective divisor $H_{12}+ H_{14}+ H_{24}+T_{11}+T_{22}$ on $X$.
Since $(X,\frac{1}{4n-3}D_{0})$ is not log canonical at $P$, we have $(X,\frac{1}{4n-3}(E+T_{33}+D'))$ is not log canonical at the point $P$. Write \[D'=a_{33}T_{33}+\Omega\]
where $a_{33}\geq 0$ and $\Omega$ is an effective $\mathbb{Q}$-Cartier divisor with $T_{33}\not\subset \support(\Omega)$. 
We have \[8n-8=D'\cdot \varphi^*(t_2)\geq a_{33}T_{33}\cdot \varphi^*(t_2)=2a_{33}\] which implies $4n-3\geq 1+a_{33}$.
The inversion of adjunction formula yields that the pair $(T_{33},\frac{1}{4n-3}(E+\Omega)|_{T_{33}})$ is not log canonical at $P$. 
Thus we obtain a contradiction that 
\[1+(4n-4)\geq ((E+\Omega)\cdot T_{33})_P\geq \textup{mult}_P((E+\Omega)|_{T_{33}})>4n-3.\]
\\

\noindent 
From Cases 1,2,3, and 4, we complete the proof of 1 of Theorem \ref{MainTheorem-3}.

    \subsection{The anti-invariant part.}
    In this part we prove that $\lct\left( X,D_{i}\right) \ge \frac{1}{4n}$ for any $D_i\in |(2n+1)K_X|_i$ when $i=1,2,3$. Without loss of generality, we may assume that $i=1$ by Lemma \ref{pluricanonical_system_bidouble_covers}. Then 
    $$D_{1} \sim (E_1 + H_{23} + H_{24} + T_{22}) + \varphi^{*}(3nl - (n+1)e_1 - (n-1)e_2 - ne_3 - ne_4).$$
Take \begin{align*}\overline{D}_1
    &=(E_1+H_{23}+H_{24}+T_{22})+
    2(2nH_{12}+nH_{34}+(n-1)E_1+(n+1)E_2)\\
    &=(2n-1)E_1+(2n+2)E_2+4nH_{12}+H_{23}+H_{24}+ 2n H_{34}+T_{22}\\
    &\in |(2n+1)K_X|_1=R_1+|\varphi^*(-(n-1)K_{Y}+L_1)|.
    \end{align*}
    \noindent
    It follows from Lemma \ref{comarelct} that  $\lct_P(X,\overline{D}_1) = \frac{1}{4n}$ where  $P\in H_{12}\cap H_{34}$.\\
    
    \noindent
    Suppose on the contrary that $(X,\frac{1}{4n}D_1)$ is not log canonical for some $D_1\in |(2n+1)K_X|_1$ at some point $P\in X$. 
    We first show that the point $\varphi(P)$ belongs to the branch locus $B=B_1 + B_2 +B_3$. Let us denote by $d_1:= \varphi(D_1)$. According to \cite[Lemma 4.1]{MR4216579}, the pair $(Y, \frac{1}{4n}d_1 + \frac{1}{2}B)$ is not log canonical at $\varphi(P)$. Assume that $\varphi(P) \notin B$. Then $(Y, \frac{1}{4n}d_1)$ is not log canonical at $\varphi(P)$. So $\glct(Y, d_1) < \frac{1}{4n}$. We notice that $d_1 \sim_{\mathbb{Q}} \frac{1}{2}(2n+1)(-K_{Y})$ and $\glct(Y, -K_{Y}) \ge \frac{1}{2}$ by \cite[Theorem 1.7]{MR2465686}. Thus $\glct(Y, d_1) \ge \frac{1}{2n+1}$ which is a contradiction. Therefore $\varphi(P) \in B = B_1 + B_2 + B_3$.
\bigskip

    It remains to show that we can obtain a contradiction if $\varphi(P) \in B = B_1 + B_2 + B_3$. For this purpose we divide the proof into a sequence of following cases:\\

    \noindent
    {\bf Case 1:} $P\in E_1 \cap H_{13}$.\\ 
    Write
    $$D_1 = a_1E_1  +a_{13}H_{13} + a_{14}H_{14}+ \Omega,$$
    \noindent
    where $a_1 \ge 1, a_{13}, a_{14} \ge 0$ and $E_1,H_{13},H_{14} \nsubseteq \support(\Omega)$ with $\Omega \geq H_{23}+H_{24}+T_{22}$. We have
    $$8n+4 = D_1\cdot\varphi^{*}(t_1) \ge 2a_1 + (H_{23} + H_{24} + T_{22})\cdot\varphi^{*}(t_1) = 2a_1 + 6.$$
    \noindent
    So $a_1 \le 4n - 1$. This implies that $(X,E_1+ \frac{1}{4n}(a_{13}H_{13} +\Omega))$ is not log canonical at $P$. Therefore 
    $$\mult_P(E_1)\mult_P(a_{13}H_{13} +\Omega) > 4n.$$
    \noindent
    We notice that
    $$E_1\cdot(a_{13}H_{13} +\Omega)=E_1\cdot(D_1 - a_1E_1  - a_{14}H_{14}) = 2n+1+a_1-a_{14}.$$
    \noindent
    Hence $$1 + a_1 - a_{14}>2n.$$
    We remark that $H_{14}\cdot\Omega \ge H_{14}\cdot (H_{23} + H_{24}+T_{22}) = 2$ and
    $$H_{14}\cdot\Omega = H_{14}\cdot(D_1 - a_1E_1  - a_{13}H_{13}-a_{14}H_{14}) = 2n+1 -a_1 +a_{14}.$$
    \noindent
    Thus $$2n \ge 1+a_1 -a_{14}.$$
    \noindent
    This leads to a contradiction that $2n\ge 1+a_1-a_{14}>2n$.\\

    \noindent
    {\bf {Case 2:} $P \in E_1\setminus  H_{13} $}.\\
    \noindent
    Write
    $$D_1 = E_1 + H_{23} + H_{24} + T_{22} + D'_1,$$
    \noindent
    where $D'_1 \in \left| \varphi^{*}(3nl - (n+1)e_1 - (n-1)e_2 - ne_3 - ne_4)\right|$. We express 
    $$D'_1 = a_1E_1 + a_{13}H_{13} + \Omega,$$
    \noindent
    where $a_1, a_{13} \ge 0$ and $E_1,H_{13}\nsubseteq \support(\Omega)$. Then
    $$8n - 4 = D'_1\cdot\varphi^{*}(t_1) \ge 2a_1.$$
    \noindent
    This implies that $1+a_1  \le 4n-1$. Since $(X, \frac{1}{4n}D_1) $ is not log canonical at $P$, $(X,E_1+ \frac{1}{4n}\Omega) $ is not log canonical at $P$. Therefore 
    $$\mult_P(E_1)\mult_P(\Omega) > 4n.$$
    \noindent
    We notice
    $$E_1\cdot\Omega = E_1\cdot(D'_1 - a_1E_1 - a_{13}H_{13} ) = 2n+2 +a_1 -a_{13}.$$
    \noindent
    From the above it follows that $2n+2 +a_1 -a_{13} > 4n$. On the other hand,
    $$0 \le H_{13}\cdot\Omega = H_{13}\cdot(D'_1 - a_1E_1 - a_{13}H_{13} ) = 2n-2 -a_1 +a_{13}.$$
    \noindent
    The two above inequalities lead to the fact that $2n- 2 < a_1 -a_{13} \le 2n-2$ which is a contradiction.\\

    \noindent
    {\bf Case 3:} $P\in H_{14} \setminus E_1$.\\
    Write 
    $$D_1 = E_1+H_{23}+H_{24}+T_{22} + D'_1,$$

    \noindent
    where $D'_1 \in \left| \varphi^{*}(3nl - (n+1)e_1 - (n-1)e_2 - ne_3 - ne_4)\right|$. Since $(X, \frac{1}{4n}D_1)$ is not log canonical at $P$, $(X, \frac{1}{4n}(H_{23} +T_{22}+D'_1))$ is not log canonical at $P$. Express 
    $$D'_1 = a_{14}H_{14} + a_1E_1 + a_4\Tilde{E}_4 + \Omega,$$
    \noindent
    where $a_{14}, a_1, a_4 \ge 0$ and $H_{14},E_1,\Tilde{E}_4 \nsubseteq \support(\Omega)$. We have
    $$8n = D'_1\cdot\varphi^*(t_3) \ge 2a_{14}.$$
    \noindent
    So $a_{14}\le 4n$. \\

    \noindent
    Assume that  $P\notin H_{14}\cap \Tilde{E}_4$. Then $(X, H_{14} + \frac{1}{4n}(H_{23} + T_{22} + \Omega))$ is not log canonical at $P$. 
    \noindent
    Thus
    \begin{align*}
        2 + 2n-2 +a_{14} -a_1 - 2a_4 &\ge \mult_P(H_{14})\mult_P(H_{23} +T_{22}  + D'_1 - a_{14}H_{14} - a_1E_1 - a_4\Tilde{E}_4)\\
        &=\mult_P(H_{14})\mult_P(H_{23} +T_{22}  + \Omega)\\
        &> 4n.
    \end{align*}

    \noindent
    Therefore one sees $a_{14} > 2n +a_1 +2a_4$.\\
    
    \noindent
    On the other hand, we have 
    $$ 0 \le \Tilde{E}_4\cdot\Omega = \Tilde{E}_4\cdot(D'_1 - a_{14}H_{14} - a_1E_1 - a_4\Tilde{E}_4)  =  4n - 2a_{14} + 4a_4. $$
    \noindent
    Thus $ a_{14} \le 2n+2a_4$. This leads to the fact that $ a_1 < 0$ which is a contradiction.\\

    \noindent
    Assume that $P\in H_{14}\cap \Tilde{E}_4$. Recall that $(X, \frac{1}{4n}(H_{23} +T_{22}+D'_1))$ is not log canonical at $P$. As $P\in H_{14}\cap \Tilde{E}_4$ and $a_{14}\leq 4n$, one has that $(X, H_{14} + \frac{1}{4n}( a_4\Tilde{E}_4 + \Omega))$ is not log canonical at $P$. So
    \begin{align*}
         2n-2 +a_{14} -a_1 &\ge \mult_P(H_{14})\mult_P(D'_1 - a_{14}H_{14} - a_1E_1 )\\
        &=\mult_P(H_{14})\mult_P(a_4\Tilde{E}_4 + \Omega)\\
        &> 4n.
    \end{align*}
    \noindent
    Therefore $a_{14} > 2n +a_1 +2$.\\
    
    \noindent
    On the other hand, one has 
    $$ 0 \le E_1\cdot\Omega = E_1\cdot(D'_1 - a_{14}H_{14} - a_1E_1 - a_4\Tilde{E}_4)  =  2n+2 - a_{14} + a_1. $$
    \noindent
    Thus $ a_{14} \le 2n+2+a_1$. This leads to the fact that $ 2n+2+a_1 < 2n+2+a_1$ which is a contradiction.\\

    \noindent
    {\bf Case 4:} $P\in E_2 \cap H_{12}$.\\    
    Write
    $$D_1 =  a_1E_1+a_2E_2+a_{12}H_{12}+a_{34}H_{34}+\Omega,$$
    \noindent
    where $a_2, a_{12}, a_{34} \ge 0, a_1\geq 1$ and $E_1, E_2,H_{12},H_{34} \nsubseteq \support(\Omega)$ with  $\Omega\geq H_{23}+H_{24}+T_{22}$. 
    We have that
    $$8n+4 = D_1\cdot\varphi^{*}(t_3) \ge 2a_{12} + (H_{23} + H_{24} + T_{22})\cdot\varphi^{*}(t_3) = 2a_{12} + 4.$$ 

    \noindent
    So $a_{12} \le 4n$. This implies that $(X,H_{12}+ \frac{1}{4n}(a_{2}E_2 +\Omega)) $ is not log canonical at $P$. Therefore 
    $$\mult_P(H_{12})\mult_P(a_{2}E_2 +\Omega) > 4n.$$
    \noindent
    We notice that
    $$H_{12}\cdot(a_{2}E_2 +\Omega)=H_{12}\cdot(D_1 -a_1E_1-a_{12}H_{12} - a_{34}H_{34}) = 2n+1-a_1+a_{12}-a_{34}.$$
    \noindent
    As $a_1\ge 1$, we obtain that $a_{12}-a_{34}>a_1+4n-(2n+1)\geq 2n$.
    
    On the other hand, $H_{34}\cdot\Omega \ge H_{34}\cdot (H_{23}+H_{24}+T_{22}) = 1$ and 
    $$2n+1=D_1\cdot H_{34}=a_{12}-a_{34}+\Omega\cdot H_{34}\geq a_{12}-a_{34}+1.$$
    \noindent
    Thus 
    $$2n+1 \geq a_{12}-a_{34}+1>2n+1.$$
    \noindent
    It is a contradiction.
    \\

    \noindent
    {\bf {Case 5:} $P \in E_2\setminus  H_{12} $}.\\
    Write
    $$D_1 = E_1 + H_{23} + H_{24} + T_{22} + D'_1,$$
    \noindent
    where $D'_1 \in \left| \varphi^{*}(3nl - (n+1)e_1 - (n-1)e_2 - ne_3 - ne_4)\right|$. We express 
    $$D'_1 = a_2E_2 + a_{12}H_{12} + a_{34}H_{34} + \Omega,$$
    \noindent
    where $a_2, a_{12}, a_{34} \ge 0$ and $E_2,H_{12},H_{34}\nsubseteq \support(\Omega)$. We have
\begin{align*}
        2n &= D'\cdot H_{12} \ge a_2 - a_{12} + a_{34} \textup{ and }\\
        2n &= D'\cdot H_{34} \ge a_{12} -a_{34}
    \end{align*}
    which yield $$a_2\leq 4n.$$
\noindent
\noindent
    So $(X,E_2 + \frac{1}{4n}(H_{23} + H_{24} +T_{22} +\Omega))$ is not log canonical at the point $P$. The inversion of adjunction formula implies that the pair $(E_2,\frac{1}{4n}(H_{23} + H_{24} +T_{22} +\Omega)|_{E_2})$ is not log canonical at $P$. Thus  $$1+(2n -2 +a_2 -a_{12}) \ge \mult_P(E_2)\mult_P(H_{23} + H_{24} +T_{22} +\Omega) > 4n $$
    \noindent
    This inequality together with the above one 
    $$2n \ge a_2 - a_{12} + a_{34} \geq a_2 - a_{12}$$
    \noindent
    implies that $2n  \ge a_2 - a_{12}  > 2n +1.$ It is a contradiction.\\
    
    \noindent
    {\bf Case 6:} $P\in (H_{12} \setminus (E_1 \cup E_2)) \cup (H_{13} \setminus (E_1 \cup E_3))$.\\
    Suppose that $P \in H_{12} \setminus (E_1 \cup E_2)$. Write 
    $$D_1 = E_1+H_{23}+H_{24}+T_{22} + D'_1,$$

    \noindent
    where $D'_1 \in \left| \varphi^{*}(3nl - (n+1)e_1 - (n-1)e_2 - ne_3 - ne_4)\right|$. Since $(X, \frac{1}{4n}D_1)$ is not log canonical at $P$, $(X, \frac{1}{4n}D'_1)$ is not log canonical at $P$. Express 
    $$D'_1 = a_{12}H_{12} + a_1E_1 + a_2E_2 + \Omega,$$
    \noindent
    where $a_{12}, a_1, a_2 \ge 0$ and $H_{12},E_1,E_2 \nsubseteq \support(\Omega)$. We have
    $$8n= D'_1\cdot\varphi^*(t_3) \ge 2a_{12}.$$
    \noindent
    So $a_{12}\le 4n$. This implies that $(X, H_{12} + \frac{1}{4n}\Omega)$ is not log canonical at $P$. 
    Thus
    \begin{align*}
        2n +a_{12} -a_1 - a_2 &\ge \mult_P(H_{12})\mult_P( D'_1 - a_{12}H_{12} - a_1E_1 - a_2E_2)\\
        &=\mult_P(H_{12})\mult_P( \Omega)\\
        &> 4n.
    \end{align*}
    \noindent
    Therefore one sees $a_{12} > 2n+a_1 +a_2$.\\

    \noindent
    On the other hand, we have 
    $$0 \le E_2\cdot\Omega = E_2\cdot(D'_1 - a_{12}H_{12} - a_1E_1 - a_2E_2) = 2n - 2 - a_{12} + a_2.$$
    \noindent
    Thus $a_{12} \le 2n -2+a_2$. This leads to the fact that $a_1 < -2$ which is a contradiction.\\

    Similar argument works in the case $P \in H_{13} \setminus (E_1 \cup E_3)$.\\

    \noindent

    \noindent
    {\bf Case 7:} $P\in E_3 \cap H_{13} $.\\
 Write
    $$D_1 =  a_1E_1+a_3E_3+a_{13}H_{13}+a_{24}H_{24}+a_{22}T_{22}+\Omega,$$
    \noindent
    where $a_3, a_{13} \ge 0, a_{1},a_{24},a_{22}\geq 1$ and $E_1, E_3, H_{13},H_{24},T_{22} \nsubseteq \support(\Omega)$ with  $\Omega\geq H_{23}.$ 
    We have that
    $$8n+4 = D_1\cdot\varphi^{*}(t_4) \ge 2a_{13} + 2a_{22}+H_{23}\cdot\varphi^{*}(t_4) = 2a_{13} + 4.$$ 

    \noindent So $a_{13} \le 4n$. This implies that $(X,H_{13}+ \frac{1}{4n}(a_{3}E_3 +\Omega)) $ is not log canonical at $P$. Therefore 
    $$\mult_P(H_{13})\mult_P(a_{3}E_3 +\Omega) > 4n.$$
    \noindent
    We notice that
    $$H_{13}\cdot(a_{3}E_3 +\Omega)=H_{13}\cdot(D_1 -a_{1}E_1- a_{13}H_{13}-a_{24}H_{24}-a_{22}T_{22}) = 2n+1-a_{1}+a_{13}-a_{24}-a_{22}.$$
    \noindent
    As $a_{22}, a_{24}\ge 1$, one has $a_{13}-a_{1}> 2n-1+a_{22}+a_{24}\ge 2n+1$.
    
    On the other hand, one has 
    $$2n+1=D_1\cdot E_{1}=- a_1+a_{13}+\Omega\cdot E_{1}\geq a_{13}-a_1$$ since $E_1\nsubseteq \support(\Omega)$.
    \noindent
    Thus  
    $$2n+1 \geq a_{13}-a_{1}>2n+1.$$
    \noindent
    It is a contradiction.\\

    \noindent
    {\bf {Case 8:} $P \in E_3\setminus  H_{13} $}.\\
    \noindent We get a contradiction similarly to Case 5.\\

    \noindent
    {\bf Case 9:} $P\in (H_{23} \setminus (E_2 \cup E_3)) \cup (H_{24} \setminus E_2 )$.\\
    \noindent
    Similarly to Case 6, we get a contradiction if $P\in (H_{23} \setminus (E_2 \cup E_3))$. Moreover we obtain a contradiction if $P\in H_{24} \setminus E_2 $ similarly to Case 3.\\

    \noindent
    {\bf Case 10:} $P\in H_{34} \setminus (E_3\cup T_{22})$.\\
    We get a contradiction analogous to Case 3.\\
    
   \noindent
   {\bf Case 11: }$P\in H_{34}\cap T_{22}$.\\ 
   \noindent
    Write 
    $$D_1 = a_1E_1+a_{12}H_{12}+a_{34}H_{34}+ a_{22}T_{22}+ \Omega,$$
    where $a_{12},a_{34}\ge 0$, $ a_1, a_{22}\ge 1$ and $E_1,H_{12},H_{34},T_{22}\nsubseteq \support(\Omega)$. Note $\Omega\ge H_{23}+H_{24}$. 
    We have $$8n+4 = D_1\cdot\varphi^*(t_1) \ge 2a_1+2a_{34}+2a_{22}+(H_{23}+H_{24})\cdot \varphi^*(t_1)\geq 2(a_{34}+a_{22})+6.$$ In particular, $a_{34}+a_{22}\leq 4n-1$. As $P\notin E_1\cup H_{12}$, one has that $(X, H_{34} + \frac{1}{4n}(a_{22}T_{22}+ \Omega))$ is not log canonical at $P$. The inversion of adjunction formula gives 
    \begin{align}
         2n+1+a_{34}-a_{12} &\ge \mult_P(H_{34})\mult_P(  D_1 -a_{12}H_{12}- a_{34}H_{34}  ) \nonumber\\
        &=\mult_P(H_{34})\mult_P(a_{22}T_{22}+\Omega)> 4n. \label{positive}
    \end{align}
    \noindent
    Therefore $a_{34} > 2n-1+a_{12}\ge 2n-1$. 


    Define $B':=a_{22}T_{22}+a_{34}H_{34} $ and $C:=\Omega$. Put $m:=\mult_P(B')=a_{22}+a_{34}(\ge 1)$ and the local intersection number $I:=(B'\cdot C)_P$ at $P$. Then $I$ is positive because $\mult_P(\Omega)>4n-a_{22}\ge1$ by (\ref{positive}) and $a_{22}\le a_{22}+a_{34}\leq 4n-1$. Note that 
   \begin{align*}
   I&=((a_{22}T_{22}+a_{34}H_{34})\cdot \Omega)_P\le \Omega\cdot (a_{22}T_{22}+a_{34}H_{34})\\
   &= (D_1-a_1E_1-a_{12}H_{12}-a_{34}H_{34}-a_{22}T_{22})\cdot (a_{22}T_{22}+a_{34}H_{34})\\
   &=(4n+2)a_{22}+(2n+1)a_{34}-a_{12}a_{34}-2a_{22}a_{34}+a_{34}^2 \\ &=(4n+2)a_{22}+(2n+1-a_{12}-2a_{22}+a_{34}) a_{34}.  
   \end{align*} 
        From \cite[Lemma 3.2]{MR4597209}, the pair $(X,\frac{1}{m}B'+\frac{m}{I}C)$ is log canonical near $P$. 
   Since $m=a_{22}+a_{34}\le 4n-1<4n$, one has $(X,\frac{1}{4n}B'+\frac{m}{I}C)$ is log canonical near $P$. 
   On the other hand,  $(X,\frac{1}{4n}D_1)$ is not log canonical at $P\in T_{22}\cap H_{34}$. So $(X,\frac{1}{4n}B'+\frac{1}{4n}C)=(X,\frac{1}{4n}(a_{22}T_{22}+a_{34}H_{34})+\frac{1}{4n}\Omega)$ is not log canonical at $P$. This implies $$\frac{a_{22}+a_{34}}{(4n+2)a_{22}+(2n+1-a_{12}-2a_{22}+a_{34}) a_{34}}\leq \frac{m}{I}<\frac{1}{4n}. $$ 
     One has 
    \begin{equation}\label{dagger}
      2a_{22}> (2n-1+a_{12}+2a_{22}-a_{34}) a_{34}.
    \end{equation}
     Note that $$2n+1=D_1\cdot H_{12}\geq a_1-a_{12}+a_{34}\geq 1-a_{12}+a_{34}.$$ 
     Then (\ref{dagger}) gives $2a_{22}\geq (2a_{22}-1)a_{34}$. 
     Recall that $a_{34}>2n-1$ and $a_{22}\ge 1$. We have $$ 2a_{22}> (2a_{22}-1)a_{34}\geq (2a_{22}-1) (2n-1),$$ which implies $n=1$. Now the inequality in (\ref{dagger}) is 
     $$2a_{22}> (1+a_{12}+2a_{22}-a_{34}) a_{34.}$$
     Since $a_{12}\ge 0$, we have 
    $$2a_{22}(a_{34}-1) < a_{34}(a_{34} - 1).$$
    Note that $$1=2n-1<a_{34}=a_{22}+a_{34}-a_{22}\leq 4n-1-a_{22}=3-a_{22}\leq 2.$$
    We obtain 
    $$2a_{22} < a_{34}\cdot\frac{a_{34} -1} {a_{34}-1}= a_{34} \leq 2,$$
     which implies $a_{22} < 1$. This contradicts to $a_{22}\ge 1$.\\

    \noindent
    {\bf Case 12:} $P\in T_{11}  $.\\
    \noindent
    If $P\in E_1\cap T_{11}$ then we get a contradiction by Case 2. Thus we may assume that $P\in T_{11}\setminus E_1$.
    
    Consider the decomposition $ T_{11}\setminus E_1 = (T_{11}\setminus (E_1 \cup H_{23} )) \bigcup (T_{11} \cap H_{23}
    )$. By Case 9, we get a contradiction if $P \in H_{23} \setminus (E_2 \cup E_3)$. In particular if $P \in T_{11} \cap H_{23}$ then we obtain a contradiction.\\

    Thus one may assume that $P\in T_{11}\setminus (E_1 \cup H_{23} )$. Write
    $$D_1 = a_{11}T_{11} + a_1E_1 + a_{23}H_{23} +\Omega,$$
    \noindent
    where $a_{1}, a_{23} \ge 1, a_{11} \ge 0$ and $T_{11},E_1,H_{23} \nsubseteq \support(\Omega)$. We have
    $$8n+4 = D_1\cdot\varphi^{*}(t_3) \ge 2a_{11}+ (H_{24} +  T_{22})\cdot\varphi^{*}(t_3) \ge 2a_{11}+4.$$
    \noindent
    So $a_{11} \le 4n$. This implies that $(X,T_{11}+ \frac{1}{4n}\Omega) $ is not log canonical at $P$. Therefore 
    $$\mult_P(T_{11})\mult_P(\Omega) > 4n.$$
    \noindent
    Since $T_{11}\cdot\Omega \ge \mult_P(T_{11})\mult_P(\Omega)$ where 
      $$  T_{11}\cdot\Omega = T_{11}\cdot(D_1 - a_{11}T_{11} - a_1E_1 - a_{23}H_{23} )
                     = 4n + 2 -a_1 - a_{23}, $$
    \noindent
    we get
    $$2 > a_1 + a_{23}  \ge 1 + 1$$
    \noindent
    which is a contradiction.\\
    
    \noindent
    {\bf Case 13: $P \in T_{33}\setminus  T_{22}$}.\\
    \noindent
    By the previous cases: 3,6,8,9 we may assume that $P \notin E_3\cup H_{12}\cup H_{14}\cup H_{24}$.
    \medskip
    
    Therefore we may assume that $P \in T_{33}\setminus  (E_3\cup H_{12}\cup H_{14}\cup H_{24}\cup T_{22})$. Write 
    $$D_1 = E_1+H_{23}+H_{24}+T_{22} + D'_1,$$

    \noindent
    where $D'_1 \in \left|\varphi^{*}(3nl-(n+1)e_1-(n-1)e_2-ne_3-ne_4)\right|.$ Since $(X, \frac{1}{4n}D_1)$ is not log canonical at $P$ where $P\notin H_{24}\cup T_{22}$, $(X, \frac{1}{4n}D'_1)$ is not log canonical at $P$. Express
    $$D'_1 = a_{33}T_{33}+\Omega,$$
    \noindent
    where $a_{33} \ge 0$ and $T_{33} \nsubseteq \support(\Omega)$. We have
    $$8n-4= D'_1.\varphi^*(t_1) \ge 2a_{33}.$$
    \noindent
    So $a_{33}\le 4n-2$. This implies that $(X, T_{33} + \frac{1}{4n}\Omega)$ is not log canonical at $P$. In particular, 
    $$\mult_P(T_{33})\mult_P(\Omega) > 4n.$$
    \noindent
    We notice
    $$T_{33}\cdot\Omega=T_{33}\cdot (D'_1 - a_{33}T_{33}) = 4n.$$
    \noindent
    Hence we obtain a contradiction.\\
    
    \noindent
{\bf {Case 14: $P \in T_{22}\cap  T_{33}$}}.\\
    \noindent
    Write
    $$D_1 = a_1E_1 +a_{14}H_{14} +a_{23}H_{23}+a_{24}H_{24}+a_{22}T_{22}+a_{33}T_{33} + \Omega,$$
    \noindent
    where $a_1, a_{22},a_{23},a_{24}\ge 1$, $a_{14}, a_{33} \ge 0$ and $E_1,H_{14}, H_{23},H_{24}, T_{22}, T_{33} \nsubseteq \support(\Omega)$. We have
    $$8n+4 = D_1\cdot\varphi^{*}(t_1) \ge 2a_{22}+2a_{33}+ (H_{23} +  H_{24})\cdot\varphi^{*}(t_1) \ge 2a_{22}+2a_{33}+4.$$
    \noindent
    So $a_{33}\leq a_{22}+a_{33} \le 4n$.  
    This implies that $(X,T_{33}+ \frac{1}{4n}(a_{22}T_{22}+\Omega)) $ is not log canonical at $P$. The inversion of adjunction formula yields  
    $$T_{33}\cdot (a_{22}T_{22}+\Omega)\ge (T_{33}\cdot (a_{22}T_{22}+\Omega))_P\ge \mult_P(T_{33})\mult_P(a_{22}T_{22}+\Omega) > 4n.$$
    \noindent
    Since $1\le a_{24}$ and 
    \begin{align*} T_{33}\cdot(a_{22}T_{22}+\Omega)&=T_{33}\cdot (D_1- a_1E_1 -a_{14}H_{14}-a_{23}H_{23}-a_{24}H_{24}-a_{33}T_{33})\\
    &=4n+2-(a_{14}+a_{24}),
    \end{align*} one has $a_{14}<1$ and $\mult_P(a_{22}T_{22}+\Omega)\le 4n+1$. As $1\le a_{22}$, it follows that $\mult_P(\Omega)=\mult_P(a_{22}T_{22}+\Omega)-a_{22}\le 4n$. 

    Since $a_{14}<1\leq a_{23}$, $1\leq a_1$ and 
    \begin{align*} 2n+1=H_{14}\cdot D_1\geq a_1-a_{14}+a_{23}+a_{22}+a_{33},
    \end{align*}
    one has $a_{22}+a_{33}<2n$. Note that both $(X,T_{22}+ \frac{1}{4n}(a_{33}T_{33}+\Omega))$ and $(X,T_{33}+ \frac{1}{4n}(a_{22}T_{22}+\Omega))$ are not log canonical at $P$. Then  the inversion of adjunction formula gives
   \begin{align*}  &2n+2\mult_P(\Omega)>a_{22}+a_{33}+2\mult_P(\Omega)\\ &=\mult_P(T_{22})\mult_P(a_{33}T_{33}+\Omega)+ \mult_P(T_{33})\mult_P(a_{22}T_{22}+\Omega)\\
   &> 4n+4n=8n   
   \end{align*}
  
   Define $B':=\Omega$ and $C:=a_{22}T_{22}+a_{33}T_{33}$. Put $m:=\mult_P(B')$ and the local intersection number $I:=(B'\cdot C)_P$ at $P$ which are positive by $\mult_P(\Omega)>3n$. From \cite[Lemma 3.2]{MR4597209}, the pair $(X,\frac{1}{m}B'+\frac{m}{I}C)$ is log canonical near $P$. 
   Since $3n< m=\mult_P(\Omega)\le 4n$ and 
   \begin{align*}
   I&=(\Omega\cdot (a_{22}T_{22}+a_{33}T_{33}))_P\le \Omega\cdot (a_{22}T_{22}+a_{33}T_{33})\\
   &\leq D_1\cdot (a_{22}T_{22}+a_{33}T_{33})=(4n+2)(a_{22}+a_{33}),
   \end{align*}
   we have that $(X,\frac{1}{4n}B'+\frac{3n}{(4n+2)(a_{22}+a_{33})}C)$ is log canonical near $P$.
   On the other hand,  $(X,\frac{1}{4n}D_1)$ is not log canonical at $P\in T_{22}\cap T_{33}$. So $(X,\frac{1}{4n}B'+\frac{1}{4n}C)=(X,\frac{1}{4n}\Omega+\frac{1}{4n}(a_{22}T_{22}+a_{33}T_{33}))$ is not log canonical at $P$. This implies $$\frac{3n}{(4n+2)(a_{22}+a_{33})}\leq \frac{m}{I}<\frac{1}{4n}.$$
   Therefore
   $$12n^2<(4n+2)(a_{22}+a_{33})<2n(4n+2)=8n^2+4n,$$ and so $n<1$. It contradicts to the assumption that $n\ge 1$.\\

    \noindent
    {\bf {Case 15:} $P \in T_{22}\setminus (H_{34} \cup T_{33})$}.\\
    \noindent
    By the previous cases, we can assume $P \in T_{22}\setminus (E_2 \cup H_{13} \cup H_{14} \cup H_{34}\cup T_{11}\cup T_{33})$.\\
    
    Let $d_1:=\varphi(D_1).$ 
    Write 
    $$d_1 = a_{22}'t_{22} + a_{23}'h_{23} +a_{24}'h_{24} + \Omega',$$ where $a_{22}',a_{23}',a_{24}'\ge 0$ and $\Omega'$ is an effective divisor on $Y$ with $t_{22},h_{23},h_{24}\nsubseteq \support(\Omega')$. Since $\varphi^*(d_1)=\varphi^*(\varphi(D_1))\geq D_1$ and $D_1\geq E_1+H_{23}+H_{24}+T_{22}$ where $\varphi^*(h_{23})=2H_{23}, \varphi^*(h_{24})=2H_{24}$ and $\varphi^*(t_{22})=2T_{22}$, one has $\min\{a_{23}', a_{24}'\}\geq \frac{1}{2}$.  Thus $$d_1\cdot t_1=(a_{22}'t_{22} + a_{23}'h_{23} +a_{24}'h_{24} + \Omega')\cdot t_1\ge a_{22}'+a_{23}'+a_{24}'\ge a_{22}'+1.$$ As   $d_1\sim_{\mathbb{Q}} -\frac{2n+1}{2}K_{Y}$, we see 
    $$d_1\cdot t_1=\frac{2n+1}{2}(3l-e_1-e_2-e_3-e_4)\cdot t_1=2n+1$$ which yields $a_{22}'\leq 2n$.
    
    Suppose now that $(X,\frac{1}{4n}D_1)$ is not log canonical at $P$. Then  $(Y, \frac{1}{4n}d_1 + \frac{1}{2}B)$ is not log canonical at $\varphi(P)$ by \cite[Lemma 4.1]{MR4216579}. The assumption $P \in T_{22}\setminus (E_2 \cup H_{13} \cup H_{14} \cup H_{34}\cup T_{11}\cup T_{33})$ implies that $(Y, \frac{1}{4n}d_1 + \frac{1}{2}t_{22})$ is not log canonical at $\varphi(P)$.  
    Since $a_{22}' \le 2n$ and $\varphi(P)\not\in h_{23}\cup h_{24}$,
    one sees that $(Y, t_{22}+\frac{1}{4n}\Omega')$ is not log canonical at $\varphi(P)$. So 
    $$\mult_{\varphi(P)}(t_{22})\mult_{\varphi(P)}(\Omega') > 4n.$$
    \noindent
    We notice that $d_1\cdot t_{22}= -\frac{2n+1}{2}K_{Y}\cdot t_2=2n+1$. 
    Then we get a contradiction that 
    \[
    2n+1=t_{22}\cdot d_1= t_{22}\cdot \Omega'\geq \mult_{\varphi(P)}(t_{22})\mult_{\varphi(P)}(\Omega') > 4n.
    \]

 \begin{Acknowledgments}
The authors are grateful to Professor Jungkai Alfred Chen for his valuable feedback on this article. The first author is supported by Vietnam Ministry of Education and Training. The second author was partially support by National Science and Technology Council of Taiwan (Grant Numbers: 110- 2115-M-008-006-MY2 and 111-2123-M-002-012-). The third author was supported by the National Research Foundation of Korea(NRF) grant funded by the Korea government(MSIT)(No. 2021R1A4A3033098) and by Basic Science Research Program through the National Research Foundation of Korea(NRF) funded by the Ministry of Education(No. RS-2023-00241086). Part of this work was done in NCTS and Quy Nhon university. We would like to thank for NCTS and Quy Nhon university for its hospitality.
\end{Acknowledgments}


\begin{paracol}{2}
\BinAddresses
  \switchcolumn
\JhengJieAddresses
  \switchcolumn
\YongJooAddresses
  \switchcolumn
\end{paracol}

\end{document}